\documentclass[a4paper,11pt,oneside]{amsart}
\pdfoutput=1
\usepackage[DIV=14,BCOR=2mm,headinclude=false,footinclude=false]{typearea}
\usepackage{pgf, tikz}
\usepackage{pgfplots}
\usetikzlibrary{shapes.geometric}
\usetikzlibrary{arrows.meta,arrows}
\usepackage[latin1]{inputenc}
\usepackage{amsfonts}
\usepackage{amsmath}
\usepackage{amsthm}
\usepackage{amssymb}
\usepackage{mathtools}
\usepackage{amsbsy, hyperref}
\usepackage{mathrsfs}
\usepackage{mathabx}
\usepackage{dsfont}
\usepackage{tcolorbox}
\usepackage{comment}
\usepackage{mathtools}
\usepackage{bbm}
\usepackage{enumerate}
\usepackage[numbers]{natbib}
\usepackage{graphicx,color}
\usepackage{yfonts}
\usepackage[labelfont=rm,format=plain,indention=0cm,singlelinecheck=off,justification=raggedright,skip=2pt]{caption}
\usepackage[fit]{truncate}
\usepackage{placeins} 
\usepackage{flafter}  
\usepackage{float}
\usepackage{tabularx}
\usepackage{ltablex}

 \usepackage{todonotes}

\usepackage[framemethod=tikz]{mdframed}
\mdfsetup{linecolor=blue,backgroundcolor=gray!10,roundcorner=10pt,leftmargin=2pt,innerleftmargin=2pt,innerrightmargin=2pt}

\newcommand{\R}{\mathbb{R}}
\newcommand{\Z}{\mathbb{Z}}

\newcommand{\N}{\mathbb{N}}

\newcommand{\ds}{\displaystyle}
\newcommand{\one}{\mathbf{1}}

\newtheorem{thmx}{Theorem}

\DeclareMathOperator{\sgn}{sgn}

\usetikzlibrary{calc,matrix,backgrounds}

\pgfdeclarelayer{overlay}
\pgfsetlayers{overlay,background,main}

\tikzset{circle/.style = {rounded corners,line width=1bp,color=#1}}%

\DeclareSymbolFont{yhlargesymbols}{OMX}{yhex}{m}{n} 
\DeclareMathAccent{\yhwidehat}{\mathord}{yhlargesymbols}{"62}

\makeatletter
\renewcommand{\tocsection}[3]{%
  \indentlabel{\@ifnotempty{#2}{\bfseries\ignorespaces#1 #2\quad}}\bfseries#3}
\renewcommand{\tocsubsection}[3]{%
  \indentlabel{\@ifnotempty{#2}{\ignorespaces#1 #2\quad}}#3}

\def\@tocline#1#2#3#4#5#6#7{\relax
  \ifnum #1>\c@tocdepth 
  \else
    \par \addpenalty\@secpenalty\addvspace{#2}%
    \begingroup \hyphenpenalty\@M
    \@ifempty{#4}{%
      \@tempdima\csname r@tocindent\number#1\endcsname\relax
    }{%
      \@tempdima#4\relax
    }%
    \parindent\z@ \leftskip#3\relax \advance\leftskip\@tempdima\relax
    \rightskip\@pnumwidth plus1em \parfillskip-\@pnumwidth
    #5\leavevmode\hskip-\@tempdima{#6}\nobreak
    \leaders\hbox{$\m@th\mkern \@dotsep mu\hbox{.}\mkern \@dotsep mu$}\hfill
    \nobreak
    \hbox to\@pnumwidth{\@tocpagenum{\ifnum#1=1\bfseries\fi#7}}\par
    \nobreak
    \endgroup
  \fi}
\AtBeginDocument{%
\expandafter\renewcommand\csname r@tocindent0\endcsname{0pt}
}
\def\l@subsection{\@tocline{2}{0pt}{2.5pc}{5pc}{}}
\makeatother

\mathtoolsset{showonlyrefs}

\newcommand{\Addresses}{{
  \bigskip
  \footnotesize
  Cristina Benea, \textsc{Nantes Universit\'e, Laboratoire de Math\'ematiques Jean Leray, Nantes 44322, France}\par\nopagebreak
 \textit{E-mail address}: \texttt{cristina.benea@univ-nantes.fr}\\
 
  Itamar Oliveira, \textsc{School of Mathematics, The Watson Building, University of Birmingham, Edgbaston, Birmingham,
B15 2TT, England}\par\nopagebreak
 \textit{E-mail addresses}: \texttt{oliveira.itamar.w@gmail.com}, \texttt{i.oliveira@bham.ac.uk}

}}

\theoremstyle{plain}
\newtheorem{theorem}{Theorem}[section]

\newtheorem{claim}[theorem]{Claim}

\theoremstyle{definition}

\newtheorem{remark}[theorem]{Remark}

\def\eqalign#1{\null\,\vcenter{\openup\jot\mathsurround\dimen12
  \ialign{\strut\hfil$\textstyle{##}$&$\textstyle{{}##}$\hfil
      \crcr#1\crcr}}\,}

\begin{document}

\begin{abstract} We prove the boundedness of a trilinear operator that is modulation invariant and which contains curvature information given by the presence of a complex exponential, adding to the small class of examples of such operators.
\end{abstract}

\title[The oscillatory biest operator]{The oscillatory biest operator}
\author{Cristina Benea and Itamar Oliveira}
\thanks{C.B. acknowledges financial support from CNRS grant JCJC 2023. I.O. was supported by ERC project FAnFArE no. 637510. He is currently supported by UK Research and Innovation (grant reference EP/W032880/1).}
\date{}
\maketitle


\section{Introduction}
Our work is motivated by a recent research direction, initiated in \cite{BBLV}, that seeks to understand operators that are both modulation invariant and which present some curvature features. The hybrid nature of these objects require a careful combination of time-frequency analysis tools with techniques traditionally used for dealing with curvature (stationary/non-stationary phase principle, $TT^*$ methods, decaying-type attributes), and in this paper we introduce an example to the yet small class of known bounded operators that have this profile.

\subsection{Modulation invariance} A representative example of a modulation-invariant operator is the bilinear Hilbert transform, defined by
\begin{equation}
BHT(f, g)(x)=\mathrm{p.v.} \int_{\R} f(x-t) g(x+t) \frac{\mathrm{d} t}{t}.
\end{equation}
The following result was proved by Lacey and Thiele \cite{LT1}, \cite{LT2}:
\begin{thmx}
\label{thm:BHT}
The bilinear Hilbert transform satisfies the following estimates:
\[
\|  BHT(f,g)  \|_q \lesssim \|f\|_{q_1} \, \|g\|_{q_2}
\]
for any $1 < q_1, q_2 \leq \infty$, $2/3<q<\infty$ with $\frac{1}{q}=\frac{1}{q_1}+\frac{1}{q_2}$.
\end{thmx}

For $a \in \R$, we denote by $M_a$ the modulation action on functions:
\[
M_af(x):= e^{2 \pi i a x} f(x).
\]
Then we can see that for any $a \in \R$, 
\[
| BHT(M_a f, M_a g)(x)|= | BHT(f, g)(x)|.
\]
One can equally see this invariance in frequency (as invariance under phase translations) since the bilinear Hilbert transform is a bilinear operator with frequency symbol $\sgn(\xi_1-\xi_2)$:
\begin{equation}
BHT(f, g)(x)= \int_{\R^2} \hat f(\xi_1) \hat g(\xi_2) \sgn(\xi_1-\xi_2) e^{2 \pi i x(\xi_1+\xi_2)} \mathrm{d} \xi_1 \mathrm{d} \xi_2.
\end{equation}

This modulation invariance makes the time-frequency analysis machinery well-suited to study $BHT$. The methods of the proof of Theorem \ref{thm:BHT} are closely related to those used to show the almost everywhere convergence of Fourier series (see \cite{Carleson}, \cite{F}), and thus require a simultaneous phase and space decomposition of the functions involved, together with a phase-space level-set analysis achieved through various stopping times.

Additionally, it can be useful to think of $BHT$ as a bilinear operator given by a frequency symbol which is singular along the line 
\begin{equation}
\Gamma_1:=\{ (\xi_1, \xi_2) : \xi_1=\xi_2   \}.
\end{equation}
Other types of multilinear operators associated to symbols singular along higher-dimensional frequency subspaces have been studied in \cite{MTT3}. The setup described in that paper however fails to cover a particularly interesting case - that of the trilinear Hilbert transform 
\begin{equation*}
THT(f, g, h)(x)=\mathrm{p.v.} \int_{\R} f(x-t) g(x+t) h(x+2t) \frac{\mathrm{d} t}{t},
\end{equation*}  
which exhibits both linear and quadratic modulation invariances. More precisely, if for $a \in \R$ and $M_{2,a}$ denotes the quadratic modulation action (that is $\ds (M_{2,a}f)(x):= e^{2\pi i a x^2} f(x)$), then one can easily check that $\ds | THT(M_{2,a} f, M_{2, 3a} g, M_{2, -a}h)(x)|= | THT(f, g, h)(x)|$ for all $a \in \R$. Its frequency symbol is given by $\ds \sgn(\xi_1-\xi_2-2 \xi_3)$, so it is singular along a non-degenerate\footnote{Notice that the plane $\ds \{ (\xi_1, \xi_2, \xi_3) : \xi_1-\xi_2-2 \xi_3 =0 \}$ can be represented as a graph over any two of the variables $\xi_1, \xi_2$ or $\xi_3$.} plane.

On the other hand, the \textit{biest operator} (which we will denote by $T$) is a trilinear operator given in frequency by  
\begin{equation}
\label{eq:def:biest}
T(f, g, h)(x) = \iiint_{\xi_1<\xi_2<\xi_3} \hat f(\xi_1) \hat g(\xi_2) \hat h(\xi_3) e^{2 \pi i x(\xi_1+\xi_2+\xi_3)} \mathrm{d}\xi_1 \mathrm{d}\xi_2 \mathrm{d}\xi_3.
\end{equation}
It was introduced in \cite{MTT2}, where a large set of $L^p$ estimates was obtained; in particular, it was shown that $T$ maps $L^{q_1} \times L^{q_2} \times L^{q_3}$ into $L^q$ for any $1<q_1, q_2, q_3 \leq \infty, 1 \leq q <\infty$, and also that it maps $L^{2} \times L^{2} \times L^{2}$ into $L^{\frac{2}{3}}$. The study of this operator was motivated by WKB expansions of eigenvalues of one-dimensional Schr\"odinger operators, in relation to \cite{ck}.

The frequency symbol of $T$ can be equally seen as $\ds \one_{\xi_1<\xi_2} \cdot \one_{\xi_2<\xi_3}$, so it is singular along the region $\ds \{ \xi_1=\xi_2 \} \cup \{\xi_2=\xi_3  \}$, which is the union of two degenerate planes in $\R^3$. As shown in \cite{MTT2}, the biest can be interpreted as a local composition of two $BHT$-type operators (one acting in $f_1$ and $f_2$, the other one in $f_2$ and $f_3$), plus a term of lower complexity. Concerning the modulation invariance, one can directly check that for any $a \in \R$,
\begin{equation}
\label{eq:mod-inv:biest}
| T(M_{a} f, M_{a} g, M_{a}h)(x)|= |T(f, g, h)(x)|.
\end{equation} 
More general cases of such trilinear operators, associated to products of symbols singular along the planes $\{ \xi_1=\xi_2 \}$ and $\{\xi_2=\xi_3 \}$ respectively, were studied in \cite{JJ}, where they were shown to satisfy the same $L^p$ estimates as the biest operator. 

\subsection{Introducing curvature features} The operator studied in the present paper, which we call the \textit{oscillatory biest operator}, is very similar to the biest operator $T$ from \eqref{eq:def:biest} (when written in spatial variables), except that the kernel now carries an oscillatory factor. More concretely, for $b>0$, the \textit{oscillatory biest} is defined by
\begin{equation}
\label{eq:def:osc:biest}
B(f_{1},f_{2},f_{3})(x)=\mathrm{p.v.}\int_{\mathbb{R}^{2}}f_{1}(x-t)f_{2}(x+t+s)f_{3}(x-s)e^{2\pi is^{b}}\frac{\mathrm{d}t}{t}\frac{\mathrm{d}s}{s}.
\end{equation}
If $b$ is not an integer, then $s^b$ should be interpreted as $\ds \sgn(s) |s|^b$. For small values of $s$ the oscillatory factor has very little effect, so in that regime $B$ will behave similarly to the classical biest operator $T$. But for large values of $s$, as we will see, the two operators will behave quite differently. Although the operator $B$ above clearly depends on the value of the parameter $b$, we suppress this dependency from the notation.

Our main theorem is
\begin{theorem}\label{mainthm}
Let $b \in (0, \infty) \setminus \{2\}$. Then the oscillatory biest operator $B$ associated to this parameter $b$ is a bounded operator from $L^{p_1} \times L^{p_2} \times L^{p_3}$ into $L^p$ for any $1<p_1, p_2, p_3<\infty$, $1 \leq p < \infty$ with $\ds \frac{1}{p_1}+\frac{1}{p_2}+\frac{1}{p_3}=\frac{1}{p}$.
\end{theorem}

Our primary motivation was that of providing new examples of operators which are both modulation invariant (the operator $B$ satisfies the same modulation invariance \eqref{eq:mod-inv:biest} as the biest) and which present curvature features. As mentioned before, this falls within the program announced in \cite{BBLV}, and its faraway goal is that of understanding how to combine tools used for studying modulation invariant operators with those used customarily for dealing with curvature/oscillatory integrals.

Up to the present time, there are only two known classes of multilinear operators that match the profile described above. In reversed chronological order, the first one is the class of $BC^a$ operators from \cite{BBLV}, defined by
\begin{equation}
\label{eq:def:BCa}
BC^a(f, g)(x):= \sup_{\lambda \in \R} \Big| \mathrm{p.v.} \int_{\R} f(x-t) g(x+t) e^{2 \pi i \lambda t^{a}} \frac{\mathrm{d}t}{t}\Big|,
\end{equation} 
where $a$ is a positive, real parameter (as above, $t^a$ should be interpreted as $\ds \sgn(t) |t|^a$ in the case when $a$ in not an integer). In \cite{BBLV} it was proved that for $a \in (0, \infty) \setminus \{ 1,2\}$, the operator $BC^a$ satisfies the same type of $L^p$ estimates as the bilinear Hilbert transform from Theorem \ref{thm:BHT}. The methods of the proof do not apply to the cases $a=1$ and $a=2$, and on a related note it should be mentioned that $BC^1$ and $BC^2$ also exhibit quadratic modulation invariance.
 
The other example is the one from \cite{cltt}, which is a bilinear Hilbert transform carrying an oscillatory factor given by the complex exponential associated to a real-valued polynomial. More precisely, if $P(x,t)$ is a real-valued polynomial, define
\begin{equation*}
BHT_P(f,g)(x):= \mathrm{p.v.} \int_{\R} f(x-t) g(x+t) e^{2 \pi i P(x,t)} \frac{\mathrm{d}t}{t}.
\end{equation*}
It was proven in \cite{cltt} that $BHT_{P}$ satisfies the same $L^p$ estimates as the bilinear Hilbert transform, uniformly with respect to the degree of the polynomial $P$. The monomials of degree less than or equal to 2 can be absorbed by the functions (in the spirit of identity \eqref{eq:osc:biest:deg1} below), so that the study of $BHT_P$ reduces to polynomials of degree $d \geq 3$, which do not contain monomials of degree less than or equal to 2.

Although there is a clear overlap between the $BC^a$ operators of \cite{BBLV} and the $BHT_P$ operators of \cite{cltt}, all we can say so far is that they offer different perspectives on multilinear operators carrying curvature information and which are modulation invariant. In fact, in \cite{cltt}, the modulation invariance is secondary in the analysis and the methods of the proof are related to the concept of $\sigma$-uniformity introduced in \cite{G}. On the other hand \cite{BBLV} combines time-frequency techniques (as in \cite{LT1}, \cite{LT2}, \cite{F}, \cite{MS2}) with the \emph{$LGC$ methodology} introduced in \cite{Lie1}.
It should also be noted that the main difficulty in the treatment of $BC^a$ operators comes from the fact that their linearized versions
\[
BC^a(f,g)(x)= \mathrm{p.v.} \int_{\R} f(x-t) g(x+t) e^{2 \pi i \lambda(x)t^{a}} \frac{\mathrm{d}t}{t}
\]
depend on a function $\lambda(x)$ which is only measurable.

Our methods of the proof are closer in spirit to those of \cite{BBLV}, although the difficulties here are of a different nature: our main task is to understand how to treat a trilinear operator which combines a bilinear Hilbert transform and an oscillatory bilinear Hilbert transform operator.

As a final remark, we should note that when $b=1$, our oscillatory biest operator reduces straightaway to the study of the biest operator. Indeed, for $b=1$,
\[
B(f_{1},f_{2},f_{3})(x)=\mathrm{p.v.}\int_{\mathbb{R}^{2}}f_{1}(x-t)f_{2}(x+t+s)f_{3}(x-s)e^{2\pi is}\frac{\mathrm{d}t}{t}\frac{\mathrm{d}s}{s}.
\]
Then we simply realize that 
\begin{equation}
\label{eq:osc:biest:deg1}
B(f_{1},f_{2},f_{3})(x)=e^{2\pi i (3a+1) x}  T(M_a f_1, M_a f_2, M_{a+1} f_3)(x),
\end{equation}
and since $T$ is bounded, so is $B$, within the same range.

Finally, we write $A \lesssim B$ so say that there exists a constant $C>0$ independent of $A$ and $B$, so that $A \leq C \cdot B$. And we write $A \sim B$ if $A \lesssim B$ and $B \lesssim A$, i.e. if the two expressions are comparable. If $A$ is much smaller than $B$, that is, if there exists a very large constant $C>10^2$ such that $A \leq B/C $, then we write $A \ll B$.

\subsection{Organization} This paper is organized as follows: in Section \ref{initialred} we split the oscillatory biest $B$ into many components according to certain space and frequency localizations. In Section \ref{sec:nonosc} we treat the non-oscillatory component of $B$, whereas in Sections \ref{sec:B_k:no:decay}, \ref{sec:B_k^1}, \ref{sec:B_k^2} and \ref{sec:B_k^3} we bound the oscillatory ones. In Section \ref{sec:finalint} we put all the pieces together to prove Theorem \ref{mainthm}.

\section{Initial reductions}\label{initialred}
Already the biest operator, when regarded in frequency, is bringing together two bilinear Hilbert transform operators, and part of the task in \cite{MTT2} was to understand how to decouple this information. Similarly, we need to understand how to decouple the bilinear Hilbert transform and the oscillatory bilinear Hilbert transform information.  For this we need to carry on several decompositions.

\subsection{Space localization}
\label{sec:initial:red:space}
To capture the effect of the monomial oscillation $e^{2\pi is^{b}}$, we begin by localizing $s$ dyadically and by writing the kernel $\frac{1}{s}$ as
\begin{equation}
\label{eq:def:rho}
\frac{1}{s}=\sum_{k \in \Z} 2^{-k}\rho(2^{-k}s),
\end{equation}
where $\rho$ is a smooth odd function supported in $\{ s: \frac{1}{2} \leq |s| \leq 2 \}$. Our trilinear operator $B(f_{1},f_{2},f_{3})$ can therefore be expressed as follows:
\begin{equation}
\label{eq:def:split:scale}
\sum_{k \in \Z} \iint_{\mathbb{R}^{2}}f_{1}(x-t)f_{2}(x+t+s)f_{3}(x-s)e^{2\pi is^{b}} 2^{-k}\rho(2^{-k}s)\frac{\mathrm{d}t}{t} \mathrm{d}s := \sum_{k \in \Z} B_{k}(f_{1},f_{2},f_{3})(x).
\end{equation}

Using Fourier inversion, this becomes
\begin{align}
\label{eq:def:op:freq}
\sum_{k \in \Z} \iiint_{\mathbb{R}^{3}}& \widehat{f_{1}}(\xi_1)\widehat{f_{2}}(\xi_2)\widehat{f_{3}}(\xi_3) \sgn(\xi_1-\xi_2) \big( \int_{\R}e^{-2 \pi is (\xi_3-\xi_2)} e^{2\pi is^{b}} 2^{-k}\rho(2^{-k}s) \mathrm{d}s \big)  e^{2\pi ix(\xi_1+\xi_2+\xi_3)} \mathrm{d}\xi_1\mathrm{d}\xi_2\mathrm{d}\xi_3.
\end{align}

For each $k \in \Z$, we proceed by applying the stationary/non-stationary phase principles to study the effect of the oscillatory integral
\[
  \int_{\R}e^{-2 \pi is (\xi_3-\xi_2)} e^{2\pi is^{b}} 2^{-k}\rho(2^{-k}s) \mathrm{d}s = \int_{\R}e^{-2 \pi i 2^k s (\xi_3-\xi_2)} e^{2\pi i 2^{kb}s^{b}}\rho(s) \mathrm{d}s. 
\]

Henceforth we denote
\begin{equation}
\label{eq:def:symb:os}
\mathfrak{m}^{b}_{k}(\zeta):= \int_{\R}e^{-2 \pi i 2^k s \zeta} e^{2\pi i 2^{kb}s^{b}}\rho(s) \mathrm{d}s
\end{equation}
and in the next section we analyze the properties of this multiplier.

\subsection{Non-oscillatory and oscillatory components} For large values of $k$, two scenarios are possible depending on whether the phase $\ds \varphi(s):= 2^{kb}s^{b}-2^k s \,\zeta$ has a critical point or not in the support of $\rho$. This boils down to applying the stationary phase or the non-stationary phase principle. If we are in the first situation, the stationary phase approximation indicates that the integral $\mathfrak{m}_k^b(\zeta)$ behaves asymptotically like
\begin{equation}
\label{eq:asympt:1}
\mathfrak{m}^{b}_{k}(\zeta) \sim \frac{1}{2^{bk/2}}e^{2\pi ic_{b}\zeta^{b/(b-1)}}\psi\left(\frac{\zeta}{2^{k(b-1)}}\right). 
\end{equation}
In the second situation, multiple integration by parts allows us to obtain more decay in $2^{-k}$. Observe that the situation in which a critical point exists in the support of $\rho$ corresponds to the case where $\ds|\zeta|\sim 2^{k(b-1)}$.

Guided by this principle, later on we will proceed with our analysis having in mind that the main contribution corresponds precisely to the situation when $|\xi_3-\xi_2|\sim 2^{k(b-1)}$.

On the other hand, for negative values of $k$, we cannot expect to have any decay since in this case the effect of the oscillation is inconsequential. Thus we will be splitting our operator into two parts:
\begin{equation}
\label{eq:def:low:high:osc}
\begin{aligned}
B(f_{1},f_{2},f_{3})(x)& =\sum_{k\leq 0}B_{k}(f_{1},f_{2},f_{3})(x) + \sum_{k > 0}B_{k}(f_{1},f_{2},f_{3})(x) \\
&:= B_0(f_{1},f_{2},f_{3})(x) + B_1(f_{1},f_{2},f_{3})(x).
\end{aligned}
\end{equation}
 
The term $B_0(f_{1},f_{2},f_{3})(x)$ is the \textit{non-oscillatory component}. It will be treated in Section \ref{sec:nonosc}, where we show that its study can be reduced to objects very close to the classical biest operator introduced in \cite{MTT2}. More precisely, we will show that $B_0(f_{1},f_{2},f_{3})(x)$ falls under the scope of the operators studied in \cite{JJ}.

The term $B_1(f_{1},f_{2},f_{3})(x)$ is the \textit{oscillatory component}, and we expect the oscillation $e^{2\pi is^{b}}$ to have nontrivial effect on it. Indeed, for all $k \geq 1$ we will study each $\ds B_k(f_{1},f_{2},f_{3})(x)$ separately (as defined above in \eqref{eq:def:split:scale}) and prove their boundedness with exponential decay in $k$ - this captures the impact of the oscillation. More precisely, we will prove the following result about the building blocks $\ds B_k(f_{1},f_{2},f_{3})(x)$ of the oscillatory component: 
\begin{theorem}
\label{thm:intermediary:osc:k:fixed}
For any exponents $p, p_1, p_2, p_3$ such that $1<p_1, p_2, p_3 < \infty$, $1 \leq p < \infty$  and $\ds \frac{1}{p_1}+\frac{1}{p_2}+\frac{1}{p_3}=\frac{1}{p}$, there exists $\delta_1=\delta_1(p_1,p_2, p_3) >0$ such that for all $f_1 \in L^{p_1}(\R), f_2 \in L^{p_2}(\R),  f_3 \in L^{p_3}(\R)$, and for all $k \in \N^*$, we have
\begin{align*}
\|B_k(f_{1},f_{2},f_{3})\|_p \lesssim 2^{- \delta_1 b k } \| f_1 \|_{p_1} \, \| f_2\|_{p_2} \, \| f_3\|_{p_3}.
\end{align*}
\end{theorem}
In order to prove this result, we need to take a closer look at each of these components.

\subsection{Further frequency decomposition of the $B_k$ operator, for $k \geq 1$}
\label{sec:localization:freq:osc} 

Now we consider $k \in \N^*$ to be fixed. We recall that $B_k$ is given by
\begin{align*}
\iiint_{\mathbb{R}^{3}}& \widehat{f_{1}}(\xi_1)\widehat{f_{2}}(\xi_2)\widehat{f_{3}}(\xi_3) \sgn(\xi_1-\xi_2) \mathfrak{m}^{b}_{k}(\xi_3-\xi_2) e^{2\pi ix(\xi_1+\xi_2+\xi_3)} \mathrm{d}\xi_1\mathrm{d}\xi_2\mathrm{d}\xi_3,
\end{align*}
and so far the functions involved are not localized in frequency (nor in space) and we still need to understand the way the symbol
\[
(\xi_1, \xi_2, \xi_3) \mapsto \sgn(\xi_1-\xi_2) \mathfrak{m}^{b}_{k}(\xi_3-\xi_2)
\] 
acts on them.
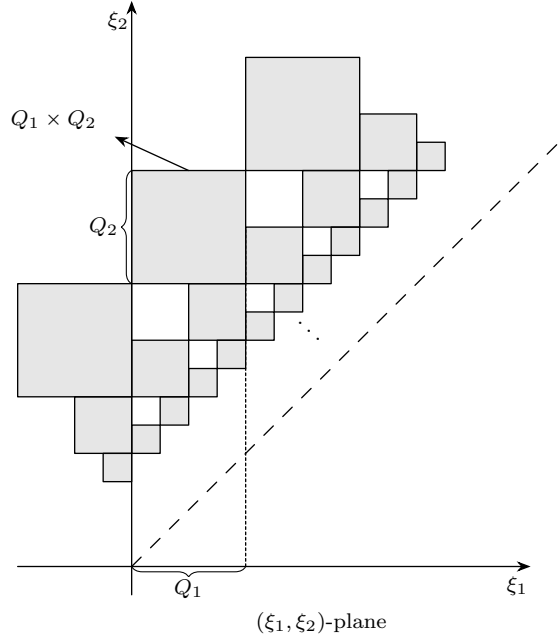
\begin{figure}[ht]
\centering
\begin{tikzpicture}[line cap=round,line join=round,>=Stealth,x=1cm,y=1cm, decoration={brace,amplitude=4pt}, scale=0.75]

\clip(-11.2,-1.5) rectangle (-1.5,11.5);

%
%

\filldraw[line width=0.5pt,color=black,fill=black,fill opacity=0.1] (-4,7) -- (-3.5,7) -- (-3.5,7.5) -- (-4,7.5) -- cycle;

\filldraw[line width=0.5pt,color=black,fill=black,fill opacity=0.1] (-4.5,6.5) -- (-4,6.5) -- (-4,7) -- (-4.5,7) -- cycle;

\filldraw[line width=0.5pt,color=black,fill=black,fill opacity=0.1] (-5,6) -- (-4.5,6) -- (-4.5,6.5) -- (-5,6.5) -- cycle;

\filldraw[line width=0.5pt,color=black,fill=black,fill opacity=0.1] (-5.5,5.5) -- (-5,5.5) -- (-5,6) -- (-5.5,6) -- cycle;

\filldraw[line width=0.5pt,color=black,fill=black,fill opacity=0.1] (-6,5) -- (-5.5,5) -- (-5.5,5.5) -- (-6,5.5) -- cycle;

\filldraw[line width=0.5pt,color=black,fill=black,fill opacity=0.1] (-6.5,4.5) -- (-6,4.5) -- (-6,5) -- (-6.5,5) -- cycle;

\filldraw[line width=0.5pt,color=black,fill=black,fill opacity=0.1] (-7,4) -- (-6.5,4) -- (-6.5,4.5) -- (-7,4.5) -- cycle;

\filldraw[line width=0.5pt,color=black,fill=black,fill opacity=0.1] (-7.5,3.5) -- (-7,3.5) -- (-7,4) -- (-7.5,4) -- cycle;

\filldraw[line width=0.5pt,color=black,fill=black,fill opacity=0.1] (-8,3) -- (-7.5,3) -- (-7.5,3.5) -- (-8,3.5) -- cycle;

\filldraw[line width=0.5pt,color=black,fill=black,fill opacity=0.1] (-8.5,2.5) -- (-8,2.5) -- (-8,3) -- (-8.5,3) -- cycle;

\filldraw[line width=0.5pt,color=black,fill=black,fill opacity=0.1] (-9,2) -- (-8.5,2) -- (-8.5,2.5) -- (-9,2.5) -- cycle;

\filldraw[line width=0.5pt,color=black,fill=black,fill opacity=0.1] (-9.5,1.5) -- (-9,1.5) -- (-9,2) -- (-9.5,2) -- cycle;

%
%

\filldraw[line width=0.5pt,color=black,fill=black,fill opacity=0.1] (-5,7) -- (-4,7) -- (-4,8) -- (-5,8) -- cycle;
\filldraw[line width=0.5pt,color=black,fill=black,fill opacity=0.1] (-6,6) -- (-5,6) -- (-5,7) -- (-6,7) -- cycle;
\filldraw[line width=0.5pt,color=black,fill=black,fill opacity=0.1] (-7,5) -- (-6,5) -- (-6,6) -- (-7,6) -- cycle;
\filldraw[line width=0.5pt,color=black,fill=black,fill opacity=0.1] (-9,3) -- (-8,3) -- (-8,4) -- (-9,4) -- cycle;
\filldraw[line width=0.5pt,color=black,fill=black,fill opacity=0.1] (-10,2) -- (-9,2) -- (-9,3) -- (-10,3) -- cycle;
\filldraw[line width=0.5pt,color=black,fill=black,fill opacity=0.1] (-8,5) -- (-7,5) -- (-7,4) -- (-8,4) -- cycle;


\filldraw[line width=0.5pt,color=black,fill=black,fill opacity=0.1] (-7,7) -- (-5,7) -- (-5,9) -- (-7,9) -- cycle;

\filldraw[line width=0.5pt,color=black,fill=black,fill opacity=0.1] (-9,5) -- (-7,5) -- (-7,7) -- (-9,7) -- cycle;

\filldraw[line width=0.5pt,color=black,fill=black,fill opacity=0.1] (-11,3) -- (-9,3) -- (-9,5) -- (-11,5) -- cycle;

%
%

%


\draw [line width=0.5pt,dash pattern=on 5pt off 5pt] (-9,0)-- (-1.5,7.5);







%
%




\draw (-11.3,7.9) node[anchor=west] {\scriptsize $Q_{1}\times Q_{2}$};
\draw [->,line width=0.5pt] (-8,7) -- (-9.3,7.6);


\draw (-7,-1) node[anchor=west] {\scriptsize $(\xi_{1},\xi_{2})$-plane};



\draw (-6.3,4.3) node[anchor=west] {\scriptsize $\ddots$};


\draw [line width=0.5pt,dash pattern=on 1pt off 1pt] (-7,0)-- (-7,6);



%
\draw [->,line width=0.5pt] (-11,0) -- (-2,0);
\draw [->,line width=0.5pt] (-9,-0.5) -- (-9,10);


\draw [decorate, color=black] (-7,0) -- (-9,0)
	node [midway, anchor=north, fill=white, inner sep=1pt, outer sep=4pt]{\scriptsize $Q_1$};
\draw [decorate, color=black] (-9,5) -- (-9,7)
	node [midway, anchor=east, fill=white, inner sep=1pt, outer sep=4pt]{\scriptsize $Q_2$};

\draw (-9.6,10) node[anchor=north west] {\scriptsize $\xi_{2}$};
\draw (-2.6,0) node[anchor=north west] {\scriptsize $\xi_{1}$};

\end{tikzpicture}
\captionsetup{justification=centering}
\caption{\footnotesize Frequency profile of the $BHT$ with Whitney squares $Q_1 \times Q_2$.} \label{figure:BHTprofile}
\end{figure}

We recall that $(\xi_1, \xi_2) \mapsto \sgn(\xi_1-\xi_2)$ is the symbol of the bilinear Hilbert transform operator, and thus it is singular along the line $\Gamma_1=\{ (\xi_1, \xi_2) : \xi_1=\xi_2\}$. The only approach that is currently known to obtain \emph{any} $L^p$ estimates for $BHT$ is to consider a Whitney decomposition of the frequency region $\R^2 \setminus \Gamma_1$, associate to it a partition of unity in order to localize the symbol to each Whitney square, and then tensorize the localized symbol through a windowed Fourier series decomposition. Even with the symbol localized and tensorized, there is not enough information in order to use orthogonality methods only. So one would need to further discretize the operator in space, thus obtaining a discrete model operator associated to phase-space \emph{Heisenberg boxes} and make good use of the phase-space disjointness of the pieces involved. We refer the reader to \cite[Chapter 6]{MS2} for more details.

One of the many building blocks of our analysis is the boundedness of $BHT$ guaranteed by Theorem \ref{thm:BHT}, which we assume throughout the manuscript. Even though we are not making use of its discretized model operator, when thinking of the bilinear Hilbert transform, the picture to have in mind is that of Figure \ref{figure:BHTprofile}.

For our analysis, it will be enough to consider the various scales making up the bilinear Hilbert transform operator. That is, we localize the frequency variable $\xi_1-\xi_2$ along frequency strips parallel to the line $\Gamma_1=\{ (\xi_1, \xi_2) : \xi_1=\xi_2\}$ and of dyadic width to obtain 
\[
\sgn(\xi_1-\xi_2)= \sum_{\ell \in \Z}\psi\big(\frac{\xi_1-\xi_2}{2^\ell}\big),
\]
where for any $\ell \in \Z$,
\begin{equation}
\label{eq:def:psi:l}
\zeta \mapsto \psi_{\ell}(\zeta):=\psi\Big(\frac{\zeta}{2^\ell}\Big)
\end{equation}
is an odd smooth function supported in the region $\ds \{ \zeta \in \R : 1/2 \leq |\zeta| \leq 2 \}$. This is equivalent to localizing dyadically the space variable $t$ from \eqref{eq:def:osc:biest}, in a way that allows to have compact support in frequency. This is mirroring the decomposition in Section \ref{sec:initial:red:space}, where we localized the variable $s$ dyadically, although the decomposition considered there has compact support in space.   

Returning to our $B_k$ operator, it becomes 
\begin{align*}
B_k(f_{1},f_{2},f_{3})(x)= \sum_{\ell \in \Z} \iiint_{\mathbb{R}^{3}}& \widehat{f_{1}}(\xi_1)\widehat{f_{2}}(\xi_2)\widehat{f_{3}}(\xi_3) \psi\big(\frac{\xi_1-\xi_2}{2^\ell}\big)\mathfrak{m}^{b}_{k}(\xi_3-\xi_2)  e^{2\pi ix(\xi_1+\xi_2+\xi_3)} \mathrm{d}\xi_1\mathrm{d}\xi_2\mathrm{d}\xi_3, 
\end{align*}
so we still need to understand the way $\psi_{\ell}(\xi_1-\xi_2)$ interacts with $\mathfrak{m}^{b}_{k}(\xi_3-\xi_2)$.

Let $N\in\mathbb{N}^{\ast}$ be a parameter to be chosen later. Guided by the stationary phase approximation in \eqref{eq:asympt:1}, we consider several regimes for the parameter $\ell$, determined by the following conditions:
\begin{itemize}
\item $\ds \mathbb{L}_1:=\{\ell \in \Z : 2^\ell \leq 2^{k(b-1)} 2^{- N b k}  \}$
\item $\ds \mathbb{L}_2:=\{\ell \in \Z : 2^{k(b-1)} 2^{- N b k} \leq 2^\ell \leq 2^{k(b-1)} 2^{ N b k}  \}$ and
\item $\ds \mathbb{L}_3:=\{\ell \in \Z : 2^{k(b-1)} 2^{N b k} \leq 2^\ell  \}$.
\end{itemize}
In other words, these regimes are determined by the scale of the bilinear Hilbert transform part, in relation to the scale of the main contribution of $\mathfrak{m}_{k}^{b}(\xi_{3}-\xi_{2})$ (see Figure \ref{figure:BigCubes2}). The sets $\mathbb{L}_j$, for $j \in \{ 1,2,3 \}$ obviously depend on the parameter $k$, but since it is fixed, we do not highlight this dependency in the notation.

\begin{figure}[ht]
\centering
\begin{tikzpicture}[line cap=round,line join=round,>=Stealth,x=1cm,y=1cm, decoration={brace,amplitude=4pt}, scale=0.75]

\clip(-11.2,-1.5) rectangle (8,11.5);

%
%

\filldraw[line width=0.5pt,color=black,fill=black,fill opacity=0.1] (-4,7) -- (-3.5,7) -- (-3.5,7.5) -- (-4,7.5) -- cycle;

\filldraw[line width=0.5pt,color=black,fill=black,fill opacity=0.1] (-4.5,6.5) -- (-4,6.5) -- (-4,7) -- (-4.5,7) -- cycle;

\filldraw[line width=0.5pt,color=black,fill=black,fill opacity=0.1] (-5,6) -- (-4.5,6) -- (-4.5,6.5) -- (-5,6.5) -- cycle;

\filldraw[line width=0.5pt,color=black,fill=black,fill opacity=0.1] (-5.5,5.5) -- (-5,5.5) -- (-5,6) -- (-5.5,6) -- cycle;

\filldraw[line width=0.5pt,color=black,fill=black,fill opacity=0.1] (-6,5) -- (-5.5,5) -- (-5.5,5.5) -- (-6,5.5) -- cycle;

\filldraw[line width=0.5pt,color=black,fill=black,fill opacity=0.1] (-6.5,4.5) -- (-6,4.5) -- (-6,5) -- (-6.5,5) -- cycle;

\filldraw[line width=0.5pt,color=black,fill=black,fill opacity=0.1] (-7,4) -- (-6.5,4) -- (-6.5,4.5) -- (-7,4.5) -- cycle;

\filldraw[line width=0.5pt,color=black,fill=black,fill opacity=0.1] (-7.5,3.5) -- (-7,3.5) -- (-7,4) -- (-7.5,4) -- cycle;

\filldraw[line width=0.5pt,color=black,fill=black,fill opacity=0.1] (-8,3) -- (-7.5,3) -- (-7.5,3.5) -- (-8,3.5) -- cycle;

\filldraw[line width=0.5pt,color=black,fill=black,fill opacity=0.1] (-8.5,2.5) -- (-8,2.5) -- (-8,3) -- (-8.5,3) -- cycle;

\filldraw[line width=0.5pt,color=black,fill=black,fill opacity=0.1] (-9,2) -- (-8.5,2) -- (-8.5,2.5) -- (-9,2.5) -- cycle;

\filldraw[line width=0.5pt,color=black,fill=black,fill opacity=0.1] (-9.5,1.5) -- (-9,1.5) -- (-9,2) -- (-9.5,2) -- cycle;

%
%

\filldraw[line width=0.5pt,color=black,fill=black,fill opacity=0.1] (-5,7) -- (-4,7) -- (-4,8) -- (-5,8) -- cycle;
\filldraw[line width=0.5pt,color=black,fill=black,fill opacity=0.1] (-6,6) -- (-5,6) -- (-5,7) -- (-6,7) -- cycle;
\filldraw[line width=0.5pt,color=black,fill=black,fill opacity=0.1] (-7,5) -- (-6,5) -- (-6,6) -- (-7,6) -- cycle;
\filldraw[line width=0.5pt,color=black,fill=black,fill opacity=0.1] (-9,3) -- (-8,3) -- (-8,4) -- (-9,4) -- cycle;
\filldraw[line width=0.5pt,color=black,fill=black,fill opacity=0.1] (-10,2) -- (-9,2) -- (-9,3) -- (-10,3) -- cycle;
\filldraw[line width=0.5pt,color=black,fill=black,fill opacity=0.1] (-8,5) -- (-7,5) -- (-7,4) -- (-8,4) -- cycle;


\filldraw[line width=0.5pt,color=black,fill=black,fill opacity=0.1] (-7,7) -- (-5,7) -- (-5,9) -- (-7,9) -- cycle;

\filldraw[line width=0.5pt,color=black,fill=black,fill opacity=0.1] (-9,5) -- (-7,5) -- (-7,7) -- (-9,7) -- cycle;

\filldraw[line width=0.5pt,color=black,fill=black,fill opacity=0.1] (-11,3) -- (-9,3) -- (-9,5) -- (-11,5) -- cycle;

%
%
\draw [->,line width=0.5pt] (-0.5,0) -- (7,0);
\draw [->,line width=0.5pt] (0,-0.5) -- (0,10);

%


\draw [line width=0.5pt,dash pattern=on 5pt off 5pt] (-9,0)-- (-1.5,7.5);


\draw [line width=0.5pt,dash pattern=on 5pt off 5pt] (0,0)-- (7.5,7.5);


\draw [line width=0.5pt,dash pattern=on 1pt off 1pt] (-0.5,1.5)-- (6.5,8.5);


\draw [line width=0.5pt,dash pattern=on 1pt off 1pt] (-1.5,2.5)-- (5.5,9.5);

%
%
\filldraw[line width=0.1pt,,dash pattern=on 0pt off 300pt,color=black,fill=black,fill opacity=0.1] (-0.3,1.7) -- (5.7,7.7) -- (4.7,8.7) -- (-1.3,2.7) -- cycle;




\draw (-11.3,7.9) node[anchor=west] {\scriptsize $Q_{1} \times Q_{2}  $};
\draw [->,line width=0.5pt] (-8,7) -- (-9.3,7.6);


\draw (-7,-1) node[anchor=west] {\scriptsize $(\xi_{1},\xi_{2})$-plane};

\draw (2,-1) node[anchor=west] {\scriptsize $(\xi_{2},\xi_{3})$-plane};


\draw (-6.3,4.3) node[anchor=west] {\scriptsize $\ddots$};


\draw [line width=0.5pt,dash pattern=on 1pt off 1pt] (-7,0)-- (-7,6);


\draw [|-|,line width=0.5pt,dash pattern=on 3pt off 3pt] (6,8) -- (5,9);
\draw (5.5,9) node[anchor=west] {\scriptsize $2^{k(b-1)} $};
\draw [|-|,line width=0.5pt,dash pattern=on 3pt off 3pt] (2,2) -- (1,3);
\draw (1.5,3) node[anchor=west] {\scriptsize $2^{k(b-1)} $};

%
\draw [->,line width=0.5pt] (-11,0) -- (-2,0);
\draw [->,line width=0.5pt] (-9,-0.5) -- (-9,10);


\draw [decorate, color=black] (-7,0) -- (-9,0)
	node [midway, anchor=north, fill=white, inner sep=1pt, outer sep=4pt]{\scriptsize $Q_1$};
\draw [decorate, color=black] (-9,5) -- (-9,7)
	node [midway, anchor=east, fill=white, inner sep=1pt, outer sep=4pt]{\scriptsize $Q_2$};

\draw (6.4,0) node[anchor=north west] {\scriptsize $\xi_{2}$};
\draw (-0.6,10) node[anchor=north west] {\scriptsize $\xi_{3}$};
\draw (-9.6,10) node[anchor=north west] {\scriptsize $\xi_{2}$};
\draw (-2.6,0) node[anchor=north west] {\scriptsize $\xi_{1}$};

\end{tikzpicture}
\captionsetup{justification=centering}
\caption{\footnotesize Display of frequency information of the symbols $\sgn(\xi_{1}-\xi_{2})$ and the moral support of $\mathfrak{m}_{k}^{b}(\xi_{3}-\xi_{2})$.} \label{figure:BigCubes2}
\end{figure}
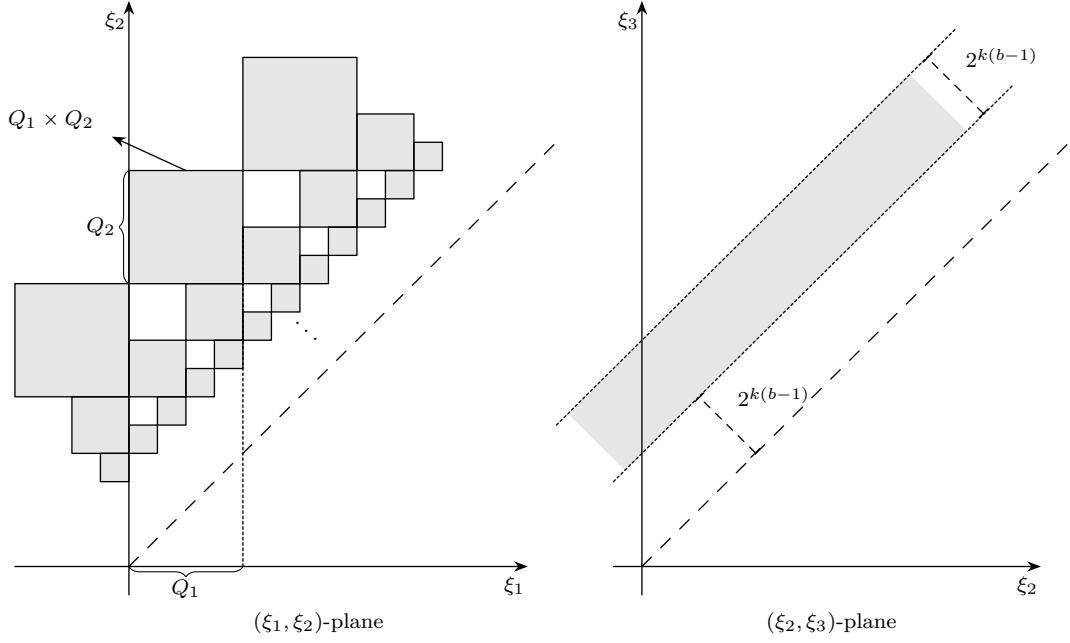

These regimes define a decomposition of $B_k(f_{1},f_{2},f_{3})(x)$ as
\begin{equation}
\label{eq:def:op:osc}
\begin{aligned}
&\sum_{\ell \in \mathbb{L}_1} \iiint_{\mathbb{R}^{3}} \widehat{f_{1}}(\xi_1)\widehat{f_{2}}(\xi_2)\widehat{f_{3}}(\xi_3) \psi\big(\frac{\xi_1-\xi_2}{2^\ell}\big)\mathfrak{m}^{b}_{k}(\xi_3-\xi_2)  e^{2\pi ix(\xi_1+\xi_2+\xi_3)} \mathrm{d}\xi_1\mathrm{d}\xi_2\mathrm{d}\xi_3 \\
 + & \sum_{\ell \in \mathbb{L}_2} \iiint_{\mathbb{R}^{3}} \widehat{f_{1}}(\xi_1)\widehat{f_{2}}(\xi_2)\widehat{f_{3}}(\xi_3) \psi\big(\frac{\xi_1-\xi_2}{2^\ell}\big)\mathfrak{m}^{b}_{k}(\xi_3-\xi_2)  e^{2\pi ix(\xi_1+\xi_2+\xi_3)} \mathrm{d}\xi_1\mathrm{d}\xi_2\mathrm{d}\xi_3 \\
 + & \sum_{\ell \in \mathbb{L}_3} \iiint_{\mathbb{R}^{3}} \widehat{f_{1}}(\xi_1)\widehat{f_{2}}(\xi_2)\widehat{f_{3}}(\xi_3) \psi\big(\frac{\xi_1-\xi_2}{2^\ell}\big)\mathfrak{m}^{b}_{k}(\xi_3-\xi_2)  e^{2\pi ix(\xi_1+\xi_2+\xi_3)} \mathrm{d}\xi_1\mathrm{d}\xi_2\mathrm{d}\xi_3 \\
 &:= B^{(1)}_k(f_{1},f_{2},f_{3})(x)+B^{(2)}_k(f_{1},f_{2},f_{3})(x)+B^{(3)}_k(f_{1},f_{2},f_{3})(x).
\end{aligned}
\end{equation}

Each of these terms will be treated separately in Sections \ref{sec:B_k^1}, \ref{sec:B_k^2}, \ref{sec:B_k^3}. The decay in $k$ claimed in Theorem \ref{thm:intermediary:osc:k:fixed} is in part a consequence of the boundedness of the following operator: for $k \in \N^*$, define
$T_k$ by
\begin{equation}
\label{def:T_k:decay}
T_k(F, G)(x):= \mathrm{p.v.} \int_{\R} F(x-t)G(x+t) \, e^{2 \pi i t^b} \rho_k(t) \mathrm{d} t,
\end{equation}
where $\rho_k(\cdot)$ is defined as above in \eqref{eq:def:rho}. Then, as a direct consequence of Theorem 2.1 in \cite{BBLV}, we have
\begin{thmx}
\label{thm:conseq:BCa}
For any exponents $q, q_1, q_2$ such that $1<q_1, q_2 \leq \infty$, $\frac{2}{3}<q<\infty$ and $\ds \frac{1}{q_1}+\frac{1}{q_2}=\frac{1}{q}$, there exists $\delta=\delta(q_1,q_2) >0$ such that for all $F \in L^{q_1}(\R), G \in L^{q_2}(\R)$, and for all $k \in \N^*$, we have
\begin{equation}
\label{eq:T_k:decay}
\|T_k(F,G)\|_{q} \lesssim 2^{-\delta b k}  \|F\|_{q_1}\,\|G\|_{q_2}.
\end{equation}
\end{thmx}

\begin{remark}
Indeed, our situation corresponds to the case when the measurable function $\lambda(x)\equiv 1$ is constant, so that the height of the oscillatory phase depends only on the parameter $t$, which is such that $\ds|t| \sim 2^k$.
\end{remark}

Theorem \ref{thm:conseq:BCa} remains of course valid if we replace $T_k$ by
\begin{equation}
\label{def:T_k:decay:2}
\tilde T_k(F, G)(x):= \mathrm{p.v.} \int_{\R} F(x-t)G\big(x+\frac{t}{2}\big) \, e^{2 \pi i t^b} \rho_k(t) \mathrm{d} t.
\end{equation}

Before moving on with the analysis of the  pieces in \eqref{eq:def:op:osc}, we first take care of the non-oscillatory part.

\section{The non-oscillatory part $B_{0}$}\label{sec:nonosc}

Here we deal with the low-oscillatory part of our operator. In other words, we show that
$$ B_0(f_{1},f_{2},f_{3})(x)= \sum_{k \leq 0} B_k(f_1, f_2, f_3)(x)$$
is close in nature to the classical biest operator from \cite{MTT2}. More precisely, we show that $B_{0}$ verifies the hypotheses of Theorem 2.2 in \cite{JJ}, and hence it is bounded within the Banach range.

If we look at $B_0$ in frequency, we realize that it writes as
\begin{align*}
B_0(f_1, f_2, f_3)(x)= \sum_{k \leq 0} \iiint_{\mathbb{R}^{3}}& \widehat{f_{1}}(\xi_1)\widehat{f_{2}}(\xi_2)\widehat{f_{3}}(\xi_3) \sgn(\xi_1-\xi_2) \mathfrak{m}^{b}_{k}(\xi_3-\xi_2) e^{2\pi ix(\xi_1+\xi_2+\xi_3)} \mathrm{d}\xi_1\mathrm{d}\xi_2\mathrm{d}\xi_3.
\end{align*}
Even though this should be similar to the trilinear operator having symbol
\[
\sgn(\xi_1-\xi_2) \sgn(\xi_3-\xi_2) \cdot \varphi(\xi_3-\xi_2), 
\]
where $\varphi$ is a smooth function supported on $[-2, 2]$, we cannot directly apply the approach used for treating biest in \cite{MTT2}. The particular form of the symbol of $T$ (the biest operator defined in \eqref{eq:def:biest}), which is given\footnote{The trilinear operator having frequency symbol $\sgn(\xi_1-\xi_2) \sgn(\xi_3-\xi_2)$ is equivalent to the one with frequency symbol $\one_{\{ \xi_1 <\xi_2 < \xi_3  \}}$, modulo terms of the form $BHT(f_1, f_2) \cdot f_3$ and $f_1 \cdot BHT(f_2, f_3)$.} by $\one_{\{ \xi_1 <\xi_2 < \xi_3  \}}$ makes the decomposition easier. This is because the symbol which is constant and equal to $1$ in the region of the frequency space where $\xi_1 <\xi_2 < \xi_3$, is still going to be equal to 1 in the regions
\[
\{ (\xi_1, \xi_2, \xi_3) : \xi_1 \ll \frac{\xi_2+\xi_3}{2}, \quad \xi_2<\xi_3  \} \text{    and    } \{ (\xi_1, \xi_2, \xi_3) : \frac{\xi_1+\xi_2}{2} \ll \xi_3 , \quad \xi_1<\xi_2  \}.
\]
These regions are important because the components of the operator restricted to them resemble local phase-space composition of two bilinear Hilbert transform operators.

So it is crucial to understand how to deal with non-constant symbols. This was studied in \cite{JJ}, where $L^p$ estimates were proven for the trilinear operator $T_{m_1 m_{2}}$ given by
\[
T_{m_1 m_2}(f_1,f_2,f_3)(x):=\int_{\mathbb{R}^{3}}m_1(\xi_1, \xi_2) m_2(\xi_2, \xi_3)\widehat{f}_{1}(\xi_{1})\widehat{f}_{2}(\xi_{2})\widehat{f}_{3}(\xi_{3})\mathrm{d}\xi_{1}\mathrm{d}\xi_{2}\mathrm{d}\xi_{3},
\] 
where $m_j(\xi_j, \xi_{j+1})$ is a two-dimensional symbol smooth away from the line $\Gamma_j:=\{ (\xi_j, \xi_{j+1}) \in \R^2 : \xi_j=\xi_{j+1} \}$, satisfying for $\alpha$ sufficiently large
\begin{equation}\label{JJ-maincond}
    |\partial^{\alpha}(m_{j}(\xi))|\lesssim\frac{1}{\textnormal{dist}(\Gamma_{j},\xi)^{|\alpha|}},\quad j\in\{1,2\}.
\end{equation}

Denote by $\textbf{D}^{\prime}$ the interior of the convex hull of the following twelve extremal points $A_{1},\ldots, A_{12}$: 

\[ \begin{array}{lll}%
A_{1}=(1,1/2,1,-3/2),  &  A_{2}=(1/2,1,1,-3/2),  &  A_{4}=(1,1/2,-3/2,1),\\
A_{4}=(1,1/2,-3/2,1),  &  A_{5}=(1,-1/2,0,1/2), &  A_{6}=(1,-1/2,1/2,0),\\
A_{7}=(1/2,-1/2,0,1),  &  A_{8} =(1/2,-1/2,1,0),  &  A_{9}=(-1/2,1,0,1/2), \\
A_{10}=(-1/2, 1, 1/2, 0),  &  A_{11}=(-1/2, 1/2, 1, 0),  &  A_{12}=(-1/2, 1/2, 0, 1).
\end{array}\]%

Additionally, denote by $\textbf{D}^{\prime\prime}$ the the open interior of the convex hull of the extremal points $\tilde{A}_{1},\ldots,\tilde{A}_{12}$, where $\tilde{A}_{j}$ is given by exchanging the 1st the 3rd coordinates of $A_{j}$. Define $\textbf{D}=\textbf{D}^{\prime}\cap\textbf{D}^{\prime\prime}$ and notice that $(1/q_1, 1/q_2, 1/q_3, 1- 1/q) \in \textbf{D}$ whenever $1<q_1, q_2, q_3\leq \infty$, $1 \leq q <\infty$ so that $1/q_1+1/q_2+1/q_3=1/q$. 

We will use the following theorem of \cite{JJ} to deal with the non-oscillatory part $B_{0}$:
\begin{thmx}[2.2 of \cite{JJ}]\label{JJ-mainthm} The operator $T_{m_1 m_2}$ maps
$$T_{m_1 m_2}: L^{q_1}\times L^{q_2}\times L^{q_3}\rightarrow L^{q},$$
provided $\ds (\frac{1}{q_1}, \frac{1}{q_2}, \frac{1}{q_3}, 1- \frac{1}{q})\in\textnormal{\textbf{D}}$.
\end{thmx}

The rest of this section is dedicated to proving the following claim.

\begin{claim} The operator $B_0$ satisfies the hypotheses of Theorem \ref{JJ-mainthm}.  
\end{claim}

\begin{proof}
From \eqref{eq:def:op:freq}, we notice that the symbol of non-oscillatory part $B_{0}$ is given by
\begin{equation}
  \sgn(\xi_1-\xi_2) \Big(\sum_{k \in \Z}  \int_{\R}e^{-2 \pi is (\xi_3-\xi_2)} e^{2\pi is^{b}} 2^{-k}\rho(2^{-k}s) \mathrm{d}s \Big).
\end{equation}
Since  $\sgn(\xi_1-\xi_2)$ is smooth away from the line $\Gamma_1$ and clearly satisfies \eqref{JJ-maincond}, it suffices to check that the symbol $m$ given by
\begin{equation}
\label{eq:m:zeta}
m(\zeta):= \sum_{k \leq 0}  \int_{\R}e^{-2 \pi is \zeta} e^{2\pi is^{b}} 2^{-k}\rho(2^{-k}s) \mathrm{d}s 
\end{equation}
is a Mikhlin symbol. More precisely, we will prove that 
\begin{equation}
\label{eq:Mikl:low:osc}
|\partial_{\zeta}^\alpha m(\zeta)| \lesssim (1+|\zeta|)^{-\alpha}
\end{equation}
for sufficiently many $\alpha$. Using the analyticity of the exponential function, we can write $m(\zeta)$ as 
\begin{align*}
\sum_{n \geq 0} \frac{(2 \pi i)^n}{n!} \sum_{k \leq 0}  \int_{\R}e^{-2 \pi is \zeta} s^{b n} 2^{-k}\rho(2^{-k}s) \mathrm{d}s = \sum_{n \geq 0} \frac{(2 \pi i)^n}{n!} \sum_{k \leq 0} 2^{kbn}  \int_{\R}e^{-2 \pi i 2^k s \zeta} s^{b n} \rho(s) \mathrm{d}s.
\end{align*}
Given the fast decay of $2^{kbn}$ for $k\leq 0$ and $n \geq 0$, it is enough to study
\begin{equation}
\label{eq:def:m_n}
m_n(\zeta):=\sum_{k \leq 0} 2^{kbn}  \big( \int_{\R}e^{-2 \pi i 2^k s \zeta} s^{b n} \rho(s) \mathrm{d}s \big).
\end{equation}
If we denote $\ds \rho_{n}(s):=s^{b n} \rho(s)$, we realize that $m_n(\zeta)$ corresponds to 
\[
m_n(\zeta)=\sum_{k \leq 0} 2^{kbn}\widehat{\rho_n}(2^k \zeta).
\]
Since $\rho$, as defined in \eqref{eq:def:rho}, is compactly supported in space, so is $\rho_n$, therefore $\rho_n$ cannot be compactly supported in frequency. Recall also that $\rho$ has average zero, which in frequency translates into $\ds \hat{\rho}(0)=0$.

\textit{Case $n=0$.} Here we have
\[
m_0(\zeta)= \sum_{k \leq 0} \hat \rho (2^k \zeta).
\]
Since $\ds \hat{\rho}(0)=0$, the mean value theorem implies that 
\[
|\hat \rho (2^k \zeta)|\lesssim 2^k |\zeta| \| (\hat \rho)'\|_\infty \lesssim 2^k |\zeta|.
\]
On the other hand, given that $\hat \rho$ is a Schwartz functions, we also have that $|\hat \rho (2^k \zeta)|\lesssim_M (2^k |\zeta|)^{-M}$ for any $M \in \mathbb{N}^*$. The decaying estimate will be particularly useful for values of $M$ large enough - a bit larger than the number of derivatives considered. Combining the two estimates that we have for $\hat \rho (2^k \zeta)$, we get
\[
|m_0(\zeta)| \lesssim  \sum_{k \leq 0} \min \big( 2^k |\zeta| , (2^k |\zeta|)^{-M} \big).
\]
If $|\zeta| \leq 1$, we directly have $\ds |m_0(\zeta)| \lesssim  \sum_{k \leq 0}  2^k |\zeta| \lesssim 1$. But if $|\zeta|>1$, then we split the sum above as
\[
|m_0(\zeta)| \lesssim  \sum_{k : 2^k \leq |\zeta|^{-1}} 2^k |\zeta| + \sum_{k: |\zeta|^{-1} \leq 2^k \leq 1} (2^k |\zeta|)^{-M} \lesssim 1. 
\]
So in either case we get $\ds |m_0(\zeta)| \lesssim 1$. The analysis is similar for the derivatives of the one-variable symbol $m_0$, provided we take $M \geq \alpha+1$:
\[
|\partial^\alpha_{\zeta} m_0(\zeta)| \lesssim  \sum_{k \leq 0} 2^{k \alpha} |(\hat \rho)^{(\alpha)}(2^k \zeta)| \lesssim \sum_{k \leq 0} 2^{k \alpha} \min(1, (2^k |\zeta|)^{-M}).
\]
If $|\zeta| \leq 1$, this implies
\[
|\partial_{\zeta}^\alpha m_0(\zeta)| \lesssim \sum_{k \leq 0} 2^{k \alpha} \lesssim 1.
\]
And if $|\zeta|>1$, for $M$ large enough we again have 
\begin{align*}
|\partial_{\zeta}^\alpha m_0(\zeta)| \lesssim & \sum_{k : 2^k \leq |\zeta|^{-1}} 2^{k\alpha} + \sum_{k: |\zeta|^{-1} \leq 2^k \leq 1} 2^{-k (M-\alpha)} |\zeta|^{-M} \\
\lesssim & |\zeta|^{-\alpha} + |\zeta|^{M-\alpha}  |\zeta|^{-M} \lesssim |\zeta|^{-\alpha}.
\end{align*}
Notice that it is enough to take $M=\alpha+1$. Thus we have proved that for any $\alpha \geq 0$,
\[
|\partial_{\zeta}^\alpha m_0(\zeta)|\lesssim (1+|\zeta|)^{-\alpha}.
\]
This means that $m_0(\zeta)$ is a Mikhlin symbol (having even bounded derivatives near the origin). 

\textit{Case $n\geq 1$.} Returning to $m_n(\zeta)$ from \eqref{eq:def:m_n}, for $n \geq 1$, we immediately see that
\[
|m_n(\zeta)| \lesssim \sum_{k \leq 0} 2^{kbn} \lesssim \frac{2^{bn}}{2^{bn}-1}. 
\]
Here we also track the dependence on $n$ for technical reasons that will be made clear later. 

Computing $\partial_{\zeta}^\alpha m_n(\zeta)$ for $\alpha \geq 0$,
\[
\partial_{\zeta}^\alpha m_n(\zeta)=\sum_{k \leq 0} 2^{kbn} (- 2 \pi i )^\alpha 2^{k \alpha} \big(  \int_{\R}e^{-2 \pi i 2^k s \zeta} s^{b n +\alpha} \rho(s) \mathrm{d}s   \big).
\]
and, integrating by parts $\alpha+1$ times and taking into account the properties of the support of $\rho$, we conclude that 
\[
|\partial_{\zeta}^\alpha m_n(\zeta)| \lesssim_{\alpha} (bn+\alpha) (bn+\alpha-1) \cdot \ldots \cdot (bn) \cdot 2^{bn+\alpha} \sum_{k \leq 0} 2^{kbn+k \alpha} \min \big( 1, 2^{-k (\alpha+1)} |\zeta|^{- (\alpha+1)} \big).
\]
We denote $\ds c_{n, \alpha}:=2^{bn+\alpha}(bn+\alpha) (bn+\alpha-1) \cdot \ldots \cdot (bn)$, which corresponds, up to constants depending only on $\alpha$, to $\ds \| \widehat{\rho_n}^{(\alpha)}\|_{\infty}$.

If $|\zeta| \leq 1$ we get 
\[
|\partial_{\zeta}^\alpha m_n(\zeta)|\lesssim_{\alpha} c_{n, \alpha} \sum_{k \leq 0} 2^{kbn+k \alpha} \lesssim c_{n, \alpha} \sum_{k \leq 0} 2^{k \alpha} \lesssim c_{n, \alpha}, 
\]
and if $|\zeta|>1$, then as before we have
\begin{align*}
|\partial_{\zeta}^\alpha m_n(\zeta)|&\lesssim_{\alpha} c_{n, \alpha}  \sum_{k: 2^{k} \leq |\zeta|^{-1}} 2^{k \alpha} + c_{n, \alpha} \sum_{k: 2^{-k} \leq |\zeta|} 2^{k \alpha} 2^{-k (\alpha+1)} |\zeta|^{- (\alpha+1)} \\
&\lesssim_{\alpha} c_{n, \alpha} |\zeta|^{-\alpha}. 
\end{align*}

Now we return to our initial symbol $m(\zeta)$. Putting everything together, we get that 
\begin{align*}
|m(\zeta)| \lesssim 1 + \sum_{n \geq 1}  \frac{(2 \pi )^n}{n!}  \frac{2^{bn}}{2^{bn}-1} \lesssim 1.
\end{align*}
And for $\alpha \geq 1$, we have
\begin{align*}
|\partial_{\zeta}^\alpha m(\zeta)| & \lesssim  (1+|\zeta|)^{-\alpha} +    \sum_{n \geq 1}  \frac{(2 \pi )^n}{n!} 2^{bn+\alpha} (bn+\alpha) (bn+\alpha-1) \cdot \ldots \cdot (bn)  (1+|\zeta|)^{-\alpha} \\
&\lesssim (1+|\zeta|)^{-\alpha}.
\end{align*}
This ends the proof of estimate \eqref{eq:Mikl:low:osc}.

Since our symbol $m(\zeta)$ is one-dimensional, notice that the number of derivatives (as required by Mikhlin-type multiplier results, such as Theorem \ref{JJ-mainthm}) for which the above estimates holds does not need to be that large. The desired bounds for $B_0$ in the Banach range follow from Theorem 2.2 in \cite{JJ}. More precisely, we have
\begin{equation}
\label{est:B_0}
\|B_0(f_1, f_2, f_3)\|_p \lesssim \prod_{j=1}^3 \|f_j\|_{p_j} \text{   for all   } 1 < p_1, p_2, p_3 \leq \infty, \, 1 \leq p <\infty. 
\end{equation}

\end{proof}

\section{Boundedness of $B_k$ without decay}
\label{sec:B_k:no:decay}

In this section we want to prove bounds without decay in $k$ for the $B_k$ operators. Of course this will not be enough for bounding $\sum_{k \geq 1} B_k$, but it is going to be important for proving Theorem \ref{thm:intermediary:osc:k:fixed} in the claimed range. This is because we can only prove some of the estimates with decay for $B_k$ in a smaller range, so interpolating with estimates without decay in a larger range allows us to obtain the result in Theorem \ref{thm:intermediary:osc:k:fixed} - see Section \ref{sec:finalint}.

The arguments used in this section only work in the Banach range, that is when the target space is an $L^p$ space with $p \geq 1$. Due to the ``entangled'' nature of our initial oscillatory biest operator $B$ and of the $B_k$ components, it is not yet clear how to overcome this obstacle.

We rewrite $B_{k}(f_{1},f_{2},f_{3})$ again as
\begin{equation*}
\begin{aligned}
    B_{k}(f_{1},f_{2},f_{3})(x)= \iiint \hat f_1(\xi_1) \hat f_2(\xi_2) \hat f_3(\xi_3) \sgn(\xi_1-\xi_2) \Big( \int_{\mathbb{R}} e^{2 \pi i s^b} e^{- 2 \pi i s (\xi_3-\xi_2)} \frac{1}{2^k} \rho(\frac{s}{2^k}) \mathrm{d}s  \Big) & \\
     e^{2 \pi i x (\xi_1+\xi_2+\xi_3)}\mathrm{d}\xi_{1}\mathrm{d}\xi_{2}\mathrm{d}\xi_{3}.\quad &
\end{aligned}     
\end{equation*}
After a change of variables and Fubini, this becomes
\begin{align*}
 \,  &\int_{\mathbb{R}} e^{2 \pi i 2^{kb} s^b}\rho(s) \Big( \iiint_{\R^3} \hat f_1(\xi_1) \hat f_2(\xi_2) e^{ 2 \pi i 2^k s \xi_2} \hat f_3(\xi_3) e^{ - 2 \pi i 2^k s \xi_3} \sgn(\xi_1-\xi_2)  e^{2 \pi i x (\xi_1+\xi_2+\xi_3)}\mathrm{d}\xi_{1}\mathrm{d}\xi_{2}\mathrm{d}\xi_{3}\Big)  ds \\
    =&\int_{\mathbb{R}} e^{2 \pi i 2^{kb} s^b} \rho(s) \cdot BHT (f_1, F_2^{2^k s})(x) \cdot F_3^{2^k s}(x) ds,
\end{align*}
where 
\[
F_2^{2^k s}(x):= f_2(x+ 2^k s) \quad \text{and} \quad F_3^{2^k s}(x):= f_3(x- 2^k s).
\]

Now if $p \geq 1$, we can use Minkowski's integral inequality and the fact that $\rho$ is supported on $\ds \{ s : 1/2 \leq |s| \leq 2\}$  to deduce
\begin{align*}
\big\|  B_k(f_{1},f_{2},f_{3})  \big\|_p & \lesssim  \int_{\mathbb{R}} \big| e^{2 \pi i 2^{kb} s^b} \rho(s) \big|  \cdot  \Big( \int_{\R} \big| BHT (f_1, F_2^{2^k s})(x) \cdot F_3^{2^k s}(x)\big|^p \mathrm{d} x \Big)^{1 \over p} \mathrm{d}s \\
 &\lesssim \int_{[-2, -{1 \over 2}]\cup [{1 \over 2}, 2]} \|f_1\|_{p_1} \cdot \|F_2^{2^k s}\|_{p_2} \cdot \|F_3^{2^k s}\|_{p_3} \mathrm{d}s \\
 & \lesssim \|f_1\|_{p_1} \cdot \|f_2\|_{p_2} \cdot \|f_3\|_{p_3}
\end{align*}
for all $p_1, p_2, p_3$ with $1 \leq p_1, p_2, p_3 \leq \infty$ and $\ds \frac{1}{p}= \frac{1}{p_1}+ \frac{1}{p_2} + \frac{1}{p_3} \leq 1$.

\vspace{.7em}
Next, we prove boundedness of $B_k$ with exponential decay in $k$, by examining separately each of the terms $B^{(1)}_k$, $B^{(2)}_k$ and $B^{(3)}_k$.

\section{The oscillatory part $B^{(1)}_k$}
\label{sec:B_k^1}

Fix $k \geq 1$. In this section we bound the trilinear operator $B^{(1)}_k$ defined in \eqref{eq:def:op:osc} by
\begin{align*}
B^{(1)}_k(f_1,f_2,f_3)(x)=\sum_{\ell \in \mathbb{L}_1} \iiint_{\mathbb{R}^{3}} \widehat{f_{1}}(\xi_1)\widehat{f_{2}}(\xi_2)\widehat{f_{3}}(\xi_3) \psi\big(\frac{\xi_1-\xi_2}{2^\ell}\big)\mathfrak{m}^{b}_{k}(\xi_3-\xi_2)  e^{2\pi ix(\xi_1+\xi_2+\xi_3)} \mathrm{d}\xi_1\mathrm{d}\xi_2\mathrm{d}\xi_3,
\end{align*}
where $\ds \mathbb{L}_1=\{\ell \in \Z : 2^\ell \leq 2^{k(b-1)} 2^{- N b k}\}$. As previously discussed in Section \ref{sec:initial:red:space}, the main contribution for $\mathfrak{m}^{b}_{k}(\xi_3-\xi_2)$ corresponds to the situation when $\ds |\xi_3-\xi_2|\sim 2^{k(b-1)}$, so the set of indices $\mathbb{L}_1$ represents, at least intuitively, a part of the frequency region 
\[
\{ (\xi_1, \xi_2, \xi_3)\in \R^3 : |\xi_1-\xi_2| \ll  |\xi_3-\xi_2|  \},
\] 
and so of the region $\ds \{ (\xi_1, \xi_2, \xi_3)\in \R^3 : |\xi_1-\xi_2| \ll  \big|\xi_3- \frac{\xi_1+\xi_2}{2}\big| \}$. Notice that this corresponds precisely to the decomposition used in the study of the biest operator, which was described in the beginning of Section \ref{sec:nonosc}.

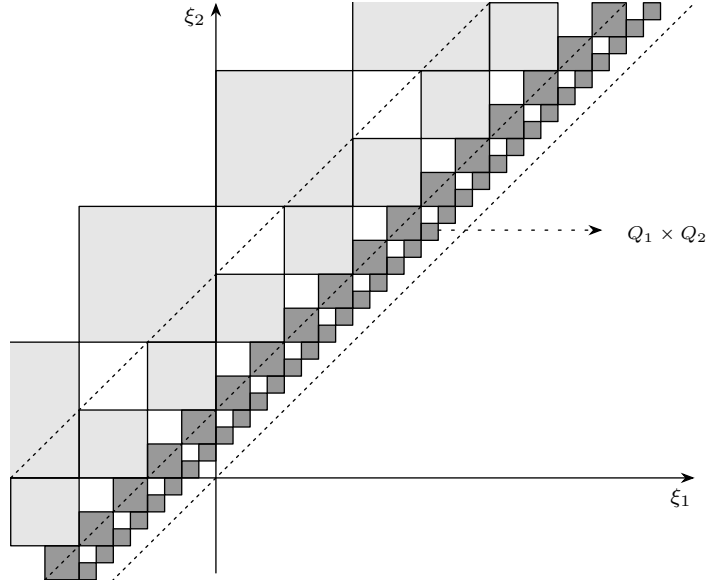
\begin{figure}[ht]
\centering
\begin{tikzpicture}[line cap=round,line join=round,>=Stealth,x=1cm,y=1cm,decoration={brace,amplitude=5pt},scale=0.9]
\clip(-4.5,-1.5) rectangle (8.5,7);

\begin{scope}
\clip(-3,-1.5) rectangle (7,15);
\filldraw[line width=0.5pt,color=black,fill=black,fill opacity=0.1] (-4,0) -- (-4,2) -- (-2,2) -- (-2,0) -- cycle;
\filldraw[line width=0.5pt,color=black,fill=black,fill opacity=0.1] (-2,2) -- (-2,4) -- (0,4) -- (0,2) -- cycle;
\filldraw[line width=0.5pt,color=black,fill=black,fill opacity=0.1] (0,4) -- (0,6) -- (2,6) -- (2,4) -- cycle;
\filldraw[line width=0.5pt,color=black,fill=black,fill opacity=0.1] (2,6) -- (2,8) -- (4,8) -- (4,6) -- cycle;
\filldraw[line width=0.5pt,color=black,fill=black,fill opacity=0.1] (4,8) -- (4,10) -- (6,10) -- (6,8) -- cycle;
\filldraw[line width=0.5pt,color=black,fill=black,fill opacity=0.1] (-3,-1) -- (-3,0) -- (-2,0) -- (-2,-1) -- cycle;
\filldraw[line width=0.5pt,color=black,fill=black,fill opacity=0.1] (-2,0) -- (-2,1) -- (-1,1) -- (-1,0) -- cycle;
\filldraw[line width=0.5pt,color=black,fill=black,fill opacity=0.1] (-1,1) -- (-1,2) -- (0,2) -- (0,1) -- cycle;

%
%

\filldraw[line width=0.5pt,color=black,fill=black,fill opacity=0.4] (1,2) -- (1,2.5) -- (1.5,2.5) -- (1.5,2) -- cycle;
\filldraw[line width=0.5pt,color=black,fill=black,fill opacity=0.4] (1.5,2.5) -- (1.5,3) -- (2,3) -- (2,2.5) -- cycle;
\filldraw[line width=0.5pt,color=black,fill=black,fill opacity=0.4] (2,3) -- (2,3.5) -- (2.5,3.5) -- (2.5,3) -- cycle;
\filldraw[line width=0.5pt,color=black,fill=black,fill opacity=0.4] (2.5,3.5) -- (2.5,4) -- (3,4) -- (3,3.5) -- cycle;
\filldraw[line width=0.5pt,color=black,fill=black,fill opacity=0.4] (1.5,2) -- (1.5,2.25) -- (1.75,2.25) -- (1.75,2) -- cycle;
\filldraw[line width=0.5pt,color=black,fill=black,fill opacity=0.4] (1.75,2.25) -- (1.75,2.5) -- (2,2.5) -- (2,2.25) -- cycle;
\filldraw[line width=0.5pt,color=black,fill=black,fill opacity=0.4] (2,2.5) -- (2,2.75) -- (2.25,2.75) -- (2.25,2.5) -- cycle;
\filldraw[line width=0.5pt,color=black,fill=black,fill opacity=0.4] (2.25,2.75) -- (2.25,3) -- (2.5,3) -- (2.5,2.75) -- cycle;
\filldraw[line width=0.5pt,color=black,fill=black,fill opacity=0.4] (2.5,3) -- (2.5,3.25) -- (2.75,3.25) -- (2.75,3) -- cycle;
\filldraw[line width=0.5pt,color=black,fill=black,fill opacity=0.4] (2.75,3.25) -- (2.75,3.5) -- (3,3.5) -- (3,3.25) -- cycle;
\filldraw[line width=0.5pt,color=black,fill=black,fill opacity=0.4] (3,3.5) -- (3,3.75) -- (3.25,3.75) -- (3.25,3.5) -- cycle;
\filldraw[line width=0.5pt,color=black,fill=black,fill opacity=0.4] (3.25,3.75) -- (3.25,4) -- (3.5,4) -- (3.5,3.75) -- cycle;
\filldraw[line width=0.5pt,color=black,fill=black,fill opacity=0.4] (-2.5,-1.5) -- (-2.5,-1) -- (-2,-1) -- (-2,-1.5) -- cycle;
\filldraw[line width=0.5pt,color=black,fill=black,fill opacity=0.4] (-2,-1) -- (-2,-0.5) -- (-1.5,-0.5) -- (-1.5,-1) -- cycle;
\filldraw[line width=0.5pt,color=black,fill=black,fill opacity=0.4] (-1.5,-0.5) -- (-1.5,0) -- (-1,0) -- (-1,-0.5) -- cycle;
\filldraw[line width=0.5pt,color=black,fill=black,fill opacity=0.4] (-1,0) -- (-1,0.5) -- (-0.5,0.5) -- (-0.5,0) -- cycle;
\filldraw[line width=0.5pt,color=black,fill=black,fill opacity=0.4] (-0.5,0.5) -- (-0.5,1) -- (0,1) -- (0,0.5) -- cycle;
\filldraw[line width=0.5pt,color=black,fill=black,fill opacity=0.4] (0,1) -- (0,1.5) -- (0.5,1.5) -- (0.5,1) -- cycle;
\filldraw[line width=0.5pt,color=black,fill=black,fill opacity=0.4] (0.5,1.5) -- (0.5,2) -- (1,2) -- (1,1.5) -- cycle;
\filldraw[line width=0.5pt,color=black,fill=black,fill opacity=0.4] (3,4) -- (3,4.5) -- (3.5,4.5) -- (3.5,4) -- cycle;
\filldraw[line width=0.5pt,color=black,fill=black,fill opacity=0.4] (3.5,4.5) -- (3.5,5) -- (4,5) -- (4,4.5) -- cycle;
\filldraw[line width=0.5pt,color=black,fill=black,fill opacity=0.4] (4,5) -- (4,5.5) -- (4.5,5.5) -- (4.5,5) -- cycle;
\filldraw[line width=0.5pt,color=black,fill=black,fill opacity=0.4] (4.5,5.5) -- (4.5,6) -- (5,6) -- (5,5.5) -- cycle;
\filldraw[line width=0.5pt,color=black,fill=black,fill opacity=0.4] (5,6) -- (5,6.5) -- (5.5,6.5) -- (5.5,6) -- cycle;
\filldraw[line width=0.5pt,color=black,fill=black,fill opacity=0.4] (5.5,6.5) -- (5.5,7) -- (6,7) -- (6,6.5) -- cycle;

\filldraw[line width=0.5pt,color=black,fill=black,fill opacity=0.4] (-2,-1.5) -- (-2,-1.25) -- (-1.75,-1.25) -- (-1.75,-1.5) -- cycle;
\filldraw[line width=0.5pt,color=black,fill=black,fill opacity=0.4] (-1.75,-1.25) -- (-1.75,-1) -- (-1.5,-1) -- (-1.5,-1.25) -- cycle;
\filldraw[line width=0.5pt,color=black,fill=black,fill opacity=0.4] (-1.5,-1) -- (-1.5,-0.75) -- (-1.25,-0.75) -- (-1.25,-1) -- cycle;
\filldraw[line width=0.5pt,color=black,fill=black,fill opacity=0.4] (-1.25,-0.75) -- (-1.25,-0.5) -- (-1,-0.5) -- (-1,-0.75) -- cycle;
\filldraw[line width=0.5pt,color=black,fill=black,fill opacity=0.4] (-1,-0.5) -- (-1,-0.25) -- (-0.75,-0.25) -- (-0.75,-0.5) -- cycle;
\filldraw[line width=0.5pt,color=black,fill=black,fill opacity=0.4] (-0.75,-0.25) -- (-0.75,0) -- (-0.5,0) -- (-0.5,-0.25) -- cycle;
\filldraw[line width=0.5pt,color=black,fill=black,fill opacity=0.4] (-0.5,0) -- (-0.5,0.25) -- (-0.25,0.25) -- (-0.25,0) -- cycle;
\filldraw[line width=0.5pt,color=black,fill=black,fill opacity=0.4] (-0.25,0.25) -- (-0.25,0.5) -- (0,0.5) -- (0,0.25) -- cycle;
\filldraw[line width=0.5pt,color=black,fill=black,fill opacity=0.4] (0,0.5) -- (0,0.75) -- (0.25,0.75) -- (0.25,0.5) -- cycle;
\filldraw[line width=0.5pt,color=black,fill=black,fill opacity=0.4] (0.25,0.75) -- (0.25,1) -- (0.5,1) -- (0.5,0.75) -- cycle;
\filldraw[line width=0.5pt,color=black,fill=black,fill opacity=0.4] (0.5,1) -- (0.5,1.25) -- (0.75,1.25) -- (0.75,1) -- cycle;
\filldraw[line width=0.5pt,color=black,fill=black,fill opacity=0.4] (0.75,1.25) -- (0.75,1.5) -- (1,1.5) -- (1,1.25) -- cycle;
\filldraw[line width=0.5pt,color=black,fill=black,fill opacity=0.4] (1,1.5) -- (1,1.75) -- (1.25,1.75) -- (1.25,1.5) -- cycle;
\filldraw[line width=0.5pt,color=black,fill=black,fill opacity=0.4] (1.25,1.75) -- (1.25,2) -- (1.5,2) -- (1.5,1.75) -- cycle;
\filldraw[line width=0.5pt,color=black,fill=black,fill opacity=0.4] (3.5,4) -- (3.5,4.25) -- (3.75,4.25) -- (3.75,4) -- cycle;
\filldraw[line width=0.5pt,color=black,fill=black,fill opacity=0.4] (3.75,4.25) -- (3.75,4.5) -- (4,4.5) -- (4,4.25) -- cycle;
\filldraw[line width=0.5pt,color=black,fill=black,fill opacity=0.4] (4,4.5) -- (4,4.75) -- (4.25,4.75) -- (4.25,4.5) -- cycle;
\filldraw[line width=0.5pt,color=black,fill=black,fill opacity=0.4] (4.25,4.75) -- (4.25,5) -- (4.5,5) -- (4.5,4.75) -- cycle;
\filldraw[line width=0.5pt,color=black,fill=black,fill opacity=0.4] (4.5,5) -- (4.5,5.25) -- (4.75,5.25) -- (4.75,5) -- cycle;
\filldraw[line width=0.5pt,color=black,fill=black,fill opacity=0.4] (4.75,5.25) -- (4.75,5.5) -- (5,5.5) -- (5,5.25) -- cycle;
\filldraw[line width=0.5pt,color=black,fill=black,fill opacity=0.4] (5,5.5) -- (5,5.75) -- (5.25,5.75) -- (5.25,5.5) -- cycle;
\filldraw[line width=0.5pt,color=black,fill=black,fill opacity=0.4] (5.25,5.75) -- (5.25,6) -- (5.5,6) -- (5.5,5.75) -- cycle;
\filldraw[line width=0.5pt,color=black,fill=black,fill opacity=0.4] (5.5,6) -- (5.5,6.25) -- (5.75,6.25) -- (5.75,6) -- cycle;
\filldraw[line width=0.5pt,color=black,fill=black,fill opacity=0.4] (5.75,6.25) -- (5.75,6.5) -- (6,6.5) -- (6,6.25) -- cycle;
\filldraw[line width=0.5pt,color=black,fill=black,fill opacity=0.4] (6,6.5) -- (6,6.75) -- (6.25,6.75) -- (6.25,6.5) -- cycle;
\filldraw[line width=0.5pt,color=black,fill=black,fill opacity=0.4] (6.25,6.75) -- (6.25,7) -- (6.5,7) -- (6.5,6.75) -- cycle;
\filldraw[line width=0.5pt,color=black,fill=black,fill opacity=0.1] (1,3) -- (1,4) -- (2,4) -- (2,3) -- cycle;
\filldraw[line width=0.5pt,color=black,fill=black,fill opacity=0.1] (0,2) -- (0,3) -- (1,3) -- (1,2) -- cycle;
\filldraw[line width=0.5pt,color=black,fill=black,fill opacity=0.1] (2,4) -- (2,5) -- (3,5) -- (3,4) -- cycle;
\filldraw[line width=0.5pt,color=black,fill=black,fill opacity=0.1] (3,5) -- (3,6) -- (4,6) -- (4,5) -- cycle;
\filldraw[line width=0.5pt,color=black,fill=black,fill opacity=0.1] (4,6) -- (4,7) -- (5,7) -- (5,6) -- cycle;

\draw [line width=0.5pt,dash pattern=on 1pt off 2pt,domain=-5:12] plot(\x,\x);
\draw [->,line width=0.5pt] (-3,0) -- (7,0);
\draw [->,line width=0.5pt] (0,-1.4) -- (0,7);
\draw [anchor=north] (6.8,0) node {\scriptsize $\xi_{1}$};
\draw [anchor= east] (0,6.8) node {\scriptsize $\xi_{2}$};
\end{scope}

\draw [line width=0.5pt,dash pattern=on 1pt off 2pt,domain=-3:6.2] plot(\x,\x+1);
\draw [line width=0.5pt,dash pattern=on 1pt off 2pt,domain=-3:7.2] plot(\x,\x+3);

\draw [->,dash pattern=on 1pt off 4pt, line width=0.5pt] (3.25,3.65) -- (5.65,3.65);
\draw [anchor=north] (6.6,3.85) node {\tiny $Q_{1}\times Q_{2}$};

\end{tikzpicture}
\captionsetup{justification=centering}
\caption{\footnotesize The highlighted cubes $Q_{1}\times Q_{2}$ have sidelength $2^\ell$ with $\ell\in\mathbb{L}_{1}$. In this case, $2^\ell$ is smaller than $2^{-k} 2^{-(N-1)bk}$.} \label{figure:X1}
\end{figure}

So now the idea is to use a Taylor expansion\footnote{While in our case it suffices to consider a first order Taylor expansion, in the situation of \cite{JJ} and \cite{M}, Taylor expansions of higher order are necessary for dealing with multilinear operators associated to product-type frequency symbols.} for $\ds \mathfrak{m}^{b}_{k}(\xi_3-\xi_2)$ in $\ds \xi_3 - \frac{\xi_1+\xi_2}{2}$. The symbol $\mathfrak{m}^{b}_{k}$ from \eqref{eq:def:symb:os} is given by an oscillatory integral expression, so taking derivatives is expected to be costly. By the Fundamental Theorem of Calculus, 
\begin{equation}
\label{eq:Taylor:1}
\begin{aligned}
\mathfrak{m}_{k}^{b}(\xi_{3}-\xi_{2}) & = \mathfrak{m}_{k}^{b}\left(\xi_{3}-\frac{\xi_{1}+\xi_{2}}{2}\right) + \int_{0}^{1}\frac{\mathrm{d}}{\mathrm{d}\tau}\mathfrak{m}_{k}^{b}\left(\xi_{3}-\frac{\xi_{1}+\xi_{2}}{2}+ \tau \big( (\xi_{3}-\xi_{2}) - (\xi_{3}-\frac{\xi_{1}+\xi_{2}}{2})   \big)\right)\mathrm{d}\tau \\
&=  \mathfrak{m}_{k}^{b}\left(\xi_{3}-\frac{\xi_{1}+\xi_{2}}{2}\right) + \int_{0}^{1} (\mathfrak{m}_{k}^{b})' \big( \xi_{3}- \frac{\xi_{1}+\xi_{2}}{2} + \tau \cdot \frac{\xi_{1}-\xi_{2}}{2}  \big)\cdot \frac{\xi_1-\xi_2}{2} \,\mathrm{d}\tau \\
&:= \mathfrak{m}_{k}^{b}\left(\xi_{3}-\frac{\xi_{1}+\xi_{2}}{2}\right) + \mathfrak{r}_{k}^{b}(\xi_1, \xi_2, \xi_3).
\end{aligned}
\end{equation} 

Replacing $\mathfrak{m}_{k}^{b}(\xi_{3}-\xi_{2})$ in our expression for $B^{(1)}_k(f_1, f_2, f_3)$, we get that
\begin{align*}
B^{(1)}_k(f_1, f_2, f_3)(x)& = \sum_{\ell \in \mathbb{L}_1} \iiint_{\mathbb{R}^{3}} \widehat{f_{1}}(\xi_1)\widehat{f_{2}}(\xi_2)\widehat{f_{3}}(\xi_3) \psi\big(\frac{\xi_1-\xi_2}{2^\ell}\big)\mathfrak{m}^{b}_{k}\big(\xi_3-\frac{\xi_{1}+\xi_{2}}{2}\big)  e^{2\pi ix(\xi_1+\xi_2+\xi_3)} \mathrm{d}\xi_1\mathrm{d}\xi_2\mathrm{d}\xi_3 \\
&+ \sum_{\ell \in \mathbb{L}_1} \iiint_{\mathbb{R}^{3}} \widehat{f_{1}}(\xi_1)\widehat{f_{2}}(\xi_2)\widehat{f_{3}}(\xi_3) \psi\big(\frac{\xi_1-\xi_2}{2^\ell}\big) \mathfrak{r}_{k}^{b}(\xi_1, \xi_2, \xi_3)  e^{2\pi ix(\xi_1+\xi_2+\xi_3)} \mathrm{d}\xi_1\mathrm{d}\xi_2\mathrm{d}\xi_3 \\
&:=\mathfrak{M}_k^{(1)}(f_1, f_2, f_3)(x)+\mathfrak{R}_k^{(1)}(f_1, f_2, f_3)(x).
\end{align*}

\subsection{The operator $\mathfrak{M}_k^{(1)}(f_1, f_2, f_3)(x)$} This corresponds to the main term of our splitting, and it can be identified with 
\begin{align}
\label{eq:id:main:1}
\tilde T_k(f_3, BHT_{\mathbb{L}_1}(f_1, f_2))(x),
\end{align}
where $\ds BHT_{\mathbb{L}_1}$ is given by 
\begin{equation}
\label{def:trunc:BHT}
 BHT_{\mathbb{L}_1}(x):=  \sum_{\ell \in \mathbb{L}_1} \iint_{\mathbb{R}^{2}} \widehat{f_{1}}(\xi_1)\widehat{f_{2}}(\xi_2) \psi\big(\frac{\xi_1-\xi_2}{2^\ell}\big) e^{2\pi ix(\xi_1+\xi_2)} \mathrm{d}\xi_1\mathrm{d}\xi_2
\end{equation}
and $\tilde T_k$ was defined in \eqref{def:T_k:decay:2}. Due to its dilation invariance, one can see after a change or variables that $BHT_{\mathbb{L}_1}$ satisfies the same $L^p$ estimates as 
\[
 BHT_{\leq 0}(x):=  \sum_{\ell' \leq 0} \iint_{\mathbb{R}^{2}} \widehat{f_{1}}(\xi_1)\widehat{f_{2}}(\xi_2) \psi\big(\frac{\xi_1-\xi_2}{2^{\ell'}}\big) e^{2\pi ix(\xi_1+\xi_2)} \mathrm{d}\xi_1\mathrm{d}\xi_2,
\]
and so $BHT_{\mathbb{L}_1}$ satisfies the same $L^p$ estimates as the bilinear Hilbert transform from Theorem \ref{thm:BHT}.

We deduce that, for any exponents $p, p_1, p_2, p_3$ such that $1<p_1, p_2, p_3 \leq \infty$, $1 \leq p < \infty$, $\ds \frac{1}{p_1}+\frac{1}{p_2}+\frac{1}{p_3}=\frac{1}{p}$, and $\ds \frac{1}{p_1}+\frac{1}{p_2}<1$, there exists $\delta=\delta(p_1,p_2, p_3) >0$ such that for all $f_1 \in L^{p_1}(\R), f_2 \in L^{p_2}(\R),  f_3 \in L^{p_3}(\R)$, and for all $k \in \N^*$, we have
\begin{align}
\label{eq:est:B^1:main}
\|\mathfrak{M}_k^{(1)}(f_1, f_2, f_3)\|_p \lesssim 2^{-\delta b k} \| f_1 \|_{p_1} \, \| f_2\|_{p_2} \, \| f_3\|_{p_3}.
\end{align}
The constraint $\ds \frac{1}{p_1}+\frac{1}{p_2}<1$ is necessary because $\ds BHT_{\mathbb{L}_1}(f_1, f_2)$ is an input for the $\tilde T_k$ operator in \eqref{eq:id:main:1}. Combining all the constraints that we have for $L^p$ estimates for $\mathfrak{M}_k^{(1)}$, we need to remove from the Banach range the endpoint $(p_1, p_1', \infty)$.

We notice that the range of boundedness of the operator $\mathfrak{M}_k^{(1)}(f_1, f_2, f_3)$ can probably be improved by making use of sizes and energies as in \cite{MTT2}, but we do not investigate this direction since we are constrained to the Banach range due to other pieces related to the operator $B_k$ - for example the remainder $\mathfrak{R}_k^{(1)}(f_1, f_2, f_3)$ below and the bound without decay that we have obtained for $B_k$ in Section \ref{sec:B_k:no:decay}.

\subsection{The remainder $\ds \mathfrak{R}_k^{(1)}(f_1, f_2, f_3)(x)$} 
Returning to the remainder operator, $\ds \mathfrak{R}_k^{(1)}(f_1, f_2, f_3)(x)$, we first need to better understand the remainder symbol $\mathfrak{r}_{k}^{b}(\xi_1, \xi_2, \xi_3)$, which, according to \eqref{eq:Taylor:1}, is equal to 
\[
\mathfrak{r}_{k}^{b}(\xi_1, \xi_2, \xi_3)= \Big( \int_{0}^{1} (\mathfrak{m}_{k}^{b})' \big( \xi_{3}- \frac{1-\tau}{2} \xi_{1} - \frac{1+\tau}{2} \xi_{2} \big) \,\mathrm{d}\tau\Big) \cdot \frac{\xi_1-\xi_2}{2}.
\]

From identity \eqref{eq:def:symb:os}, we have
\begin{align*}
(\mathfrak{m}^{b}_{k})'(\zeta) & = \frac{\mathrm{d} }{\mathrm{d} \zeta}\Big( \int_{\R}e^{-2 \pi i 2^k s \zeta} e^{2\pi i 2^{kb}s^{b}}\rho(s) \mathrm{d}s \Big) \\
&= (-2\pi i 2^k ) \Big( \int_{\R}e^{-2 \pi i 2^k s \zeta} e^{2\pi i 2^{kb}s^{b}} s\,\rho(s) \mathrm{d}s \Big).
\end{align*}

Now we replace this in our formula for $\mathfrak{R}_k^{(1)}$ to get
\begin{align*}
\mathfrak{R}_k^{(1)}(f_1, f_2, f_3)(x)&=  \sum_{\ell \in \mathbb{L}_1}  \iiint_{\mathbb R^3} \widehat{f_{1}}(\xi_1)\widehat{f_{2}}(\xi_2)\widehat{f_{3}}(\xi_3) \widehat{f_{4}}(\xi_4) \psi\big(\frac{\xi_1-\xi_2}{2^\ell}\big)  \frac{\xi_1-\xi_2}{2} e^{2 \pi i x (\xi_1+\xi_2+\xi_3)}  \\
& (-2\pi i 2^k )  \Big(  \int_{0}^{1}    \int_{\R}e^{-2 \pi i 2^k s \big( \xi_{3}- \frac{1-\tau}{2} \xi_{1} - \frac{1+\tau}{2} \xi_{2}  \big)} e^{2\pi i 2^{kb}s^{b}} s\,\rho(s) \mathrm{d}s  \mathrm{d}\tau  \Big) \mathrm{d}\xi_1\mathrm{d}\xi_2 \mathrm{d}\xi_3.
\end{align*}

Recalling that the function $\rho$ is supported on $\ds \{ s \in \R: \frac{1}{2} \leq |s| \leq 2  \}$, and after pulling out the integrals in $s$ and $\tau$, the above expression becomes
\begin{equation*}
\begin{aligned}
\mathfrak{R}_k^{(1)}(f_1, f_2, f_3)(x)=(-\pi i) \int_{0}^{1} \int_{\{ |s| \sim 1  \}} e^{2\pi i 2^{kb}s^{b}} s \,\rho(s) \Big[ \sum_{\ell \in \mathbb{L}_1} 2^{\ell+k}\iiint_{\mathbb R^3} \widehat{f_{1}}(\xi_1) e^{2 \pi i 2^k s \frac{1-\tau}{2} \xi_1} \, \widehat{f_{2}}(\xi_2) e^{2 \pi i 2^k s \frac{1+\tau}{2} \xi_2} & \\
 \widehat{f_{3}}(\xi_3) e^{-2 \pi i 2^k s \xi_{3}} \widehat{f_{3}}(\xi_3)  \psi\big(\frac{\xi_1-\xi_2}{2^\ell}\big)  \frac{\xi_1-\xi_2}{2^{\ell}} e^{2 \pi i x (\xi_1+\xi_2+\xi_3)} \mathrm{d}\xi_1\mathrm{d}\xi_2 \mathrm{d}\xi_3 \Big] \mathrm{d}s  \mathrm{d}\tau.&
\end{aligned}
\end{equation*}

For $1 \leq j \leq 3$, we let $F_j^{k,s,\tau}$ be defined by: 
\begin{align*}
\widehat{F_1^{k,s,\tau}}(\xi_1)=  \widehat{f_{1}}(\xi_1) e^{2 \pi i 2^k s \frac{1-\tau}{2} \xi_1}, \qquad \widehat{F_2^{k,s,\tau}}(\xi_2)=  \widehat{f_{2}}(\xi_2) e^{2 \pi i 2^k s \frac{1+\tau}{2} \xi_2}, \qquad \widehat {F_3^{k,s,\tau}}(\xi_3)=  \widehat{f_{3}}(\xi_3) e^{-2 \pi i 2^k s \xi_{3}},
\end{align*}
although $F_3^{k,s,\tau}$ does not really depend on $\tau$. These frequency modulations correspond to space translations, thus we have 
\[
\|F_j^{k,s,\tau}\|_{p_j}= \|f_j\|_{p_j} \qquad \text{for all  } 0<p_j\leq \infty,  1 \leq j \leq 3.
\]

For $\ell \in \mathbb L_1$, we also define $\tilde B_\ell$ to be the bilinear operator given by
\begin{align*}
\tilde B_\ell(G_1,G_2)(x):= \iint_{\R^2} \widehat{G_{1}}(\eta_1)\, \widehat{G_{2}}(\eta_2)   \psi\big(\frac{\eta_1-\eta_2}{2^\ell}\big)  \cdot \frac{\eta_1-\eta_2}{2^{\ell}} \, e^{2\pi ix(\eta_1+\eta_2)} \mathrm{d}\eta_1\mathrm{d}\eta_2,
\end{align*}
which corresponds to a fixed-scale $BHT$-type operator. It is known (see for example \cite[Chapter 6]{MS2}) that $\tilde B_\ell$ satisfies better $L^p$ estimates than the $BHT$ operator itself; more precisely, $\tilde B_\ell$ is bounded from $L^{q_1} \times L^{q_2} \to L^q$ for any $1\leq q_1, q_2\leq \infty$, $1/2 \leq q <\infty$, with bounds that are independent of the scale parameter $\ell \in \mathbb{L}_1$.

With this, we realize that 
\begin{equation*}
\begin{aligned}
\mathfrak{R}_k^{(1)}(f_1, f_2, f_3)(x)= \sum_{\ell \in \mathbb{L}_1} 2^{\ell+k} \cdot (-\pi i)  \int_{0}^{1} \int_{\{ |s| \sim 1  \}} e^{2\pi i 2^{kb}s^{b}} s \,\rho(s) \tilde B_{\ell}(F_1^{k,s,\tau}, F_2^{k,s,\tau})(x) \cdot F_3^{k,s,\tau}(x)\mathrm{d}s  \mathrm{d}\tau.&
\end{aligned}
\end{equation*}

Now if we want to estimate the $L^p$ norm of $\mathfrak{R}_k(f_1, f_2, f_3)$, for $1 \leq p <\infty$, we can invoke Minkowski's inequality and triangle inequality to obtain, via H\"older and the boundedness of $\tilde B_\ell$, that
\begin{equation*}
\begin{aligned}
\big\|\mathfrak{R}_k^{(1)}(f_1, f_2, f_3)\big\|_p \lesssim &\sum_{\ell \in \mathbb{L}_1} 2^{\ell+k} \int_{0}^{1} \int_{\{ |s| \sim 1  \}} \big\|   \tilde B_{\ell}(F_1^{k,s,\tau}, F_2^{k,s,\tau}) \big\|_{\frac{p_1 p_2}{p_1+p_2}} \cdot \| F_3^{k,s,\tau}  \|_{p_3} \mathrm{d}s  \mathrm{d}\tau \\
\lesssim & \sum_{\ell \in \mathbb{L}_1} 2^{\ell+k} \|f_1\|_{p_1} \|f_2\|_{p_2} \|f_2\|_{p_3}
\end{aligned}
\end{equation*}
as long as 
\[
0< \frac{1}{p_1}+\frac{1}{p_2}+\frac{1}{p_3}=\frac{1}{p} \leq 1, \quad 1 \leq p_1, p_2, p_3 \leq \infty.
\] 
Now if we recall that $\mathbb{L}_1=\{\ell \in \Z : 2^\ell \leq 2^{k(b-1)} 2^{- N b k}\}$, we realize that we have just proved that
\begin{align*}
\big\|\mathfrak{R}_k^{(1)}(f_1, f_2, f_3)\big\|_p \lesssim 2^{-kb (N-1)}  \|f_1\|_{p_1} \|f_2\|_{p_2} \|f_2\|_{p_3}
\end{align*} 
in the above-mentioned range. So if $N>1$, we obtain an estimate of the same flavor as \eqref{eq:est:B^1:main}.

Although the reminder term $\mathfrak{R}_k^{(1)}$ has more decay than the main term $\mathfrak{M}_k^{(1)}$, overall we get that for any exponents $1<p_1, p_2, p_3 \leq \infty$, $1 \leq p < \infty$, $1/{p_1}+1/{p_2}+1/{p_3}=1/{p}$ with $(p_1, p_2, p_3) \neq (p_1, p_1', \infty)$, there exists $\delta=\delta(p_1, p_2, p_3)$ such that
\begin{align}
\label{eq:est:B^1_k}
\|B_k^{(1)}(f_1, f_2, f_3)\|_p \lesssim 2^{-\delta b k} \| f_1 \|_{p_1} \, \| f_2\|_{p_2} \, \| f_3\|_{p_3}.
\end{align}

\section{The oscillatory part $B^{(2)}_k$}
\label{sec:B_k^2}
Here we study the operator
\begin{align*}
B^{(2)}_k(f_1, f_2, f_3)(x):=\sum_{\ell \in \mathbb{L}_1} \iiint_{\mathbb{R}^{3}} \widehat{f_{1}}(\xi_1)\widehat{f_{2}}(\xi_2)\widehat{f_{3}}(\xi_3) \psi\big(\frac{\xi_1-\xi_2}{2^\ell}\big)\mathfrak{m}^{b}_{k}(\xi_3-\xi_2)  e^{2\pi ix(\xi_1+\xi_2+\xi_3)} \mathrm{d}\xi_1\mathrm{d}\xi_2\mathrm{d}\xi_3,
\end{align*}
where $\ds \mathbb{L}_2:=\{\ell \in \Z : 2^{k(b-1)} 2^{- N b k} \leq 2^\ell \leq 2^{k(b-1)} 2^{ N b k}\}$.

We notice that the set $\ds \mathbb{L}_2$ has cardinality $2Nb k$, so if we manage to prove for every $\ell \in \mathbb{L}_2$ that the operator 
\begin{align*}
B^{(2)}_{k, \ell}(f_1, f_2, f_3)(x):=\iiint_{\mathbb{R}^{3}} \widehat{f_{1}}(\xi_1)\widehat{f_{2}}(\xi_2)\widehat{f_{3}}(\xi_3) \psi\big(\frac{\xi_1-\xi_2}{2^\ell}\big)\mathfrak{m}^{b}_{k}(\xi_3-\xi_2)  e^{2\pi ix(\xi_1+\xi_2+\xi_3)} \mathrm{d}\xi_1\mathrm{d}\xi_2\mathrm{d}\xi_3
\end{align*}
is bounded with exponential decay in $k$, then we are done. This is because the exponential decay in $k$ gains over the loss that comes from the cardinality of $\mathbb{L}_2$. We prove this result in a smaller range, but thanks to the interpolation with the estimates for $B_k$ without decay (from Section \ref{sec:B_k:no:decay}) which are available in the whole Banach range, we obtain the boundedness of $B_k$ in the range of Theorem \ref{thm:intermediary:osc:k:fixed}.

In order to prove our desired estimates, we will tensorize the symbol $\ds \psi\big(\frac{\xi_1-\xi_2}{2^\ell}\big)$. Consider a covering of the region $\ds \{ (\xi_1, \xi_2) \in \R^2 : \big| \xi_1-\xi_2 \big| \sim 2^\ell \}$ by squares having sidelength $\sim 2^\ell$, and denote by $\mathcal W_{\ell}$ this collection of squares. Next, we consider a partition of unity associated to it, consisting of bump functions compactly supported on dilates of these squares. With this, $\psi\big(\frac{\xi_1-\xi_2}{2^\ell}\big)$ becomes
\begin{align*}
\psi\big(\frac{\xi_1-\xi_2}{2^\ell}\big) = \sum_{Q_1 \times Q_2 \in \mathcal{W}_\ell} \phi_{Q_1 \times Q_2}(\xi_1, \xi_2) \psi\big(\frac{\xi_1-\xi_2}{2^\ell}\big).
\end{align*} 
Now we use Fourier series in order to tensorize the information associated to the square $Q_1 \times Q_2$ given by the term
\[
\phi_{Q_1 \times Q_2}(\xi_1, \xi_2) \psi\big(\frac{\xi_1-\xi_2}{2^\ell}\big).
\]
By means of this, $\psi\big(\frac{\xi_1-\xi_2}{2^\ell}\big)$ becomes
\[
\psi\big(\frac{\xi_1-\xi_2}{2^\ell}\big)=\sum_{Q=Q_1 \times Q_2 \in \mathcal W_{\ell}}\sum_{n\in\mathbb{Z}^{2}}C_{n}^Q \, \Phi_{Q_{1},n}(\xi_{1})\Phi_{Q_{2},n}(\xi_{2}),
\]
where the Fourier coefficients $C_{n}^Q$ are given by
\begin{align*}
C_{n}^Q=\frac{1}{|Q_1||Q_2|} \int_{\mathbb R^2} \phi_{Q_1 \times Q_2}(\eta_1, \eta_2) \psi\big(\frac{\eta_1-\eta_2}{2^\ell}\big) e^{-2 \pi i (\frac{\eta_1}{|Q_1|}, \frac{\eta_2}{|Q_2|})\cdot n} \mathrm{d} \eta_1 \mathrm{d} \eta_2. 
\end{align*}
For more details about such decompositions, see \cite[Chapter 6]{MS2}.

Since the Fourier coefficients have arbitrary polynomial decay as $|n| \to \infty$, we can regard $\psi\big(\frac{\xi_1-\xi_2}{2^\ell}\big)$ as a superposition of terms of the form
\begin{align*}
 \sum_{Q_1 \times Q_2 \in \mathcal{W}_\ell} \phi_{Q_1}(\xi_1) \, \phi_{Q_2}(\xi_2).
\end{align*}

\begin{figure}[ht]
\centering
\begin{tikzpicture}[line cap=round,line join=round,>=Stealth,x=1cm,y=1cm,decoration={brace,amplitude=5pt},scale=0.9]
\clip(-4.5,-1.5) rectangle (8.5,7);

\begin{scope}
\clip(-3,-1.5) rectangle (7,15);
\filldraw[line width=0.5pt,color=black,fill=black,fill opacity=0.1] (-4,0) -- (-4,2) -- (-2,2) -- (-2,0) -- cycle;
\filldraw[line width=0.5pt,color=black,fill=black,fill opacity=0.1] (-2,2) -- (-2,4) -- (0,4) -- (0,2) -- cycle;
\filldraw[line width=0.5pt,color=black,fill=black,fill opacity=0.1] (0,4) -- (0,6) -- (2,6) -- (2,4) -- cycle;
\filldraw[line width=0.5pt,color=black,fill=black,fill opacity=0.1] (2,6) -- (2,8) -- (4,8) -- (4,6) -- cycle;
\filldraw[line width=0.5pt,color=black,fill=black,fill opacity=0.1] (4,8) -- (4,10) -- (6,10) -- (6,8) -- cycle;

%
%
\filldraw[line width=0.5pt,color=black,fill=black,fill opacity=0.4] (-3,-1) -- (-3,0) -- (-2,0) -- (-2,-1) -- cycle;
\filldraw[line width=0.5pt,color=black,fill=black,fill opacity=0.4] (-2,0) -- (-2,1) -- (-1,1) -- (-1,0) -- cycle;
\filldraw[line width=0.5pt,color=black,fill=black,fill opacity=0.4] (-1,1) -- (-1,2) -- (0,2) -- (0,1) -- cycle;
\filldraw[line width=0.5pt,color=black,fill=black,fill opacity=0.4] (1,3) -- (1,4) -- (2,4) -- (2,3) -- cycle;
\filldraw[line width=0.5pt,color=black,fill=black,fill opacity=0.4] (0,2) -- (0,3) -- (1,3) -- (1,2) -- cycle;
\filldraw[line width=0.5pt,color=black,fill=black,fill opacity=0.4] (2,4) -- (2,5) -- (3,5) -- (3,4) -- cycle;
\filldraw[line width=0.5pt,color=black,fill=black,fill opacity=0.4] (3,5) -- (3,6) -- (4,6) -- (4,5) -- cycle;
\filldraw[line width=0.5pt,color=black,fill=black,fill opacity=0.4] (4,6) -- (4,7) -- (5,7) -- (5,6) -- cycle;

\filldraw[line width=0.5pt,color=black,fill=black,fill opacity=0.1] (1,2) -- (1,2.5) -- (1.5,2.5) -- (1.5,2) -- cycle;
\filldraw[line width=0.5pt,color=black,fill=black,fill opacity=0.1] (1.5,2.5) -- (1.5,3) -- (2,3) -- (2,2.5) -- cycle;
\filldraw[line width=0.5pt,color=black,fill=black,fill opacity=0.1] (2,3) -- (2,3.5) -- (2.5,3.5) -- (2.5,3) -- cycle;
\filldraw[line width=0.5pt,color=black,fill=black,fill opacity=0.1] (2.5,3.5) -- (2.5,4) -- (3,4) -- (3,3.5) -- cycle;
\filldraw[line width=0.5pt,color=black,fill=black,fill opacity=0.1] (1.5,2) -- (1.5,2.25) -- (1.75,2.25) -- (1.75,2) -- cycle;
\filldraw[line width=0.5pt,color=black,fill=black,fill opacity=0.1] (1.75,2.25) -- (1.75,2.5) -- (2,2.5) -- (2,2.25) -- cycle;
\filldraw[line width=0.5pt,color=black,fill=black,fill opacity=0.1] (2,2.5) -- (2,2.75) -- (2.25,2.75) -- (2.25,2.5) -- cycle;
\filldraw[line width=0.5pt,color=black,fill=black,fill opacity=0.1] (2.25,2.75) -- (2.25,3) -- (2.5,3) -- (2.5,2.75) -- cycle;
\filldraw[line width=0.5pt,color=black,fill=black,fill opacity=0.1] (2.5,3) -- (2.5,3.25) -- (2.75,3.25) -- (2.75,3) -- cycle;
\filldraw[line width=0.5pt,color=black,fill=black,fill opacity=0.1] (2.75,3.25) -- (2.75,3.5) -- (3,3.5) -- (3,3.25) -- cycle;
\filldraw[line width=0.5pt,color=black,fill=black,fill opacity=0.1] (3,3.5) -- (3,3.75) -- (3.25,3.75) -- (3.25,3.5) -- cycle;
\filldraw[line width=0.5pt,color=black,fill=black,fill opacity=0.1] (3.25,3.75) -- (3.25,4) -- (3.5,4) -- (3.5,3.75) -- cycle;
\filldraw[line width=0.5pt,color=black,fill=black,fill opacity=0.1] (-2.5,-1.5) -- (-2.5,-1) -- (-2,-1) -- (-2,-1.5) -- cycle;
\filldraw[line width=0.5pt,color=black,fill=black,fill opacity=0.1] (-2,-1) -- (-2,-0.5) -- (-1.5,-0.5) -- (-1.5,-1) -- cycle;
\filldraw[line width=0.5pt,color=black,fill=black,fill opacity=0.1] (-1.5,-0.5) -- (-1.5,0) -- (-1,0) -- (-1,-0.5) -- cycle;
\filldraw[line width=0.5pt,color=black,fill=black,fill opacity=0.1] (-1,0) -- (-1,0.5) -- (-0.5,0.5) -- (-0.5,0) -- cycle;
\filldraw[line width=0.5pt,color=black,fill=black,fill opacity=0.1] (-0.5,0.5) -- (-0.5,1) -- (0,1) -- (0,0.5) -- cycle;
\filldraw[line width=0.5pt,color=black,fill=black,fill opacity=0.1] (0,1) -- (0,1.5) -- (0.5,1.5) -- (0.5,1) -- cycle;
\filldraw[line width=0.5pt,color=black,fill=black,fill opacity=0.1] (0.5,1.5) -- (0.5,2) -- (1,2) -- (1,1.5) -- cycle;
\filldraw[line width=0.5pt,color=black,fill=black,fill opacity=0.1] (3,4) -- (3,4.5) -- (3.5,4.5) -- (3.5,4) -- cycle;
\filldraw[line width=0.5pt,color=black,fill=black,fill opacity=0.1] (3.5,4.5) -- (3.5,5) -- (4,5) -- (4,4.5) -- cycle;
\filldraw[line width=0.5pt,color=black,fill=black,fill opacity=0.1] (4,5) -- (4,5.5) -- (4.5,5.5) -- (4.5,5) -- cycle;
\filldraw[line width=0.5pt,color=black,fill=black,fill opacity=0.1] (4.5,5.5) -- (4.5,6) -- (5,6) -- (5,5.5) -- cycle;
\filldraw[line width=0.5pt,color=black,fill=black,fill opacity=0.1] (5,6) -- (5,6.5) -- (5.5,6.5) -- (5.5,6) -- cycle;
\filldraw[line width=0.5pt,color=black,fill=black,fill opacity=0.1] (5.5,6.5) -- (5.5,7) -- (6,7) -- (6,6.5) -- cycle;

\filldraw[line width=0.5pt,color=black,fill=black,fill opacity=0.1] (-2,-1.5) -- (-2,-1.25) -- (-1.75,-1.25) -- (-1.75,-1.5) -- cycle;
\filldraw[line width=0.5pt,color=black,fill=black,fill opacity=0.1] (-1.75,-1.25) -- (-1.75,-1) -- (-1.5,-1) -- (-1.5,-1.25) -- cycle;
\filldraw[line width=0.5pt,color=black,fill=black,fill opacity=0.1] (-1.5,-1) -- (-1.5,-0.75) -- (-1.25,-0.75) -- (-1.25,-1) -- cycle;
\filldraw[line width=0.5pt,color=black,fill=black,fill opacity=0.1] (-1.25,-0.75) -- (-1.25,-0.5) -- (-1,-0.5) -- (-1,-0.75) -- cycle;
\filldraw[line width=0.5pt,color=black,fill=black,fill opacity=0.1] (-1,-0.5) -- (-1,-0.25) -- (-0.75,-0.25) -- (-0.75,-0.5) -- cycle;
\filldraw[line width=0.5pt,color=black,fill=black,fill opacity=0.1] (-0.75,-0.25) -- (-0.75,0) -- (-0.5,0) -- (-0.5,-0.25) -- cycle;
\filldraw[line width=0.5pt,color=black,fill=black,fill opacity=0.1] (-0.5,0) -- (-0.5,0.25) -- (-0.25,0.25) -- (-0.25,0) -- cycle;
\filldraw[line width=0.5pt,color=black,fill=black,fill opacity=0.1] (-0.25,0.25) -- (-0.25,0.5) -- (0,0.5) -- (0,0.25) -- cycle;
\filldraw[line width=0.5pt,color=black,fill=black,fill opacity=0.1] (0,0.5) -- (0,0.75) -- (0.25,0.75) -- (0.25,0.5) -- cycle;
\filldraw[line width=0.5pt,color=black,fill=black,fill opacity=0.1] (0.25,0.75) -- (0.25,1) -- (0.5,1) -- (0.5,0.75) -- cycle;
\filldraw[line width=0.5pt,color=black,fill=black,fill opacity=0.1] (0.5,1) -- (0.5,1.25) -- (0.75,1.25) -- (0.75,1) -- cycle;
\filldraw[line width=0.5pt,color=black,fill=black,fill opacity=0.1] (0.75,1.25) -- (0.75,1.5) -- (1,1.5) -- (1,1.25) -- cycle;
\filldraw[line width=0.5pt,color=black,fill=black,fill opacity=0.1] (1,1.5) -- (1,1.75) -- (1.25,1.75) -- (1.25,1.5) -- cycle;
\filldraw[line width=0.5pt,color=black,fill=black,fill opacity=0.1] (1.25,1.75) -- (1.25,2) -- (1.5,2) -- (1.5,1.75) -- cycle;
\filldraw[line width=0.5pt,color=black,fill=black,fill opacity=0.1] (3.5,4) -- (3.5,4.25) -- (3.75,4.25) -- (3.75,4) -- cycle;
\filldraw[line width=0.5pt,color=black,fill=black,fill opacity=0.1] (3.75,4.25) -- (3.75,4.5) -- (4,4.5) -- (4,4.25) -- cycle;
\filldraw[line width=0.5pt,color=black,fill=black,fill opacity=0.1] (4,4.5) -- (4,4.75) -- (4.25,4.75) -- (4.25,4.5) -- cycle;
\filldraw[line width=0.5pt,color=black,fill=black,fill opacity=0.1] (4.25,4.75) -- (4.25,5) -- (4.5,5) -- (4.5,4.75) -- cycle;
\filldraw[line width=0.5pt,color=black,fill=black,fill opacity=0.1] (4.5,5) -- (4.5,5.25) -- (4.75,5.25) -- (4.75,5) -- cycle;
\filldraw[line width=0.5pt,color=black,fill=black,fill opacity=0.1] (4.75,5.25) -- (4.75,5.5) -- (5,5.5) -- (5,5.25) -- cycle;
\filldraw[line width=0.5pt,color=black,fill=black,fill opacity=0.1] (5,5.5) -- (5,5.75) -- (5.25,5.75) -- (5.25,5.5) -- cycle;
\filldraw[line width=0.5pt,color=black,fill=black,fill opacity=0.1] (5.25,5.75) -- (5.25,6) -- (5.5,6) -- (5.5,5.75) -- cycle;
\filldraw[line width=0.5pt,color=black,fill=black,fill opacity=0.1] (5.5,6) -- (5.5,6.25) -- (5.75,6.25) -- (5.75,6) -- cycle;
\filldraw[line width=0.5pt,color=black,fill=black,fill opacity=0.1] (5.75,6.25) -- (5.75,6.5) -- (6,6.5) -- (6,6.25) -- cycle;
\filldraw[line width=0.5pt,color=black,fill=black,fill opacity=0.1] (6,6.5) -- (6,6.75) -- (6.25,6.75) -- (6.25,6.5) -- cycle;
\filldraw[line width=0.5pt,color=black,fill=black,fill opacity=0.1] (6.25,6.75) -- (6.25,7) -- (6.5,7) -- (6.5,6.75) -- cycle;

\draw [line width=0.5pt,dash pattern=on 1pt off 2pt,domain=-5:12] plot(\x,\x);
\draw [->,line width=0.5pt] (-3,0) -- (7,0);
\draw [->,line width=0.5pt] (0,-1.4) -- (0,7);
\draw [anchor=north] (6.8,0) node {\scriptsize $\xi_{1}$};
\draw [anchor= east] (0,6.8) node {\scriptsize $\xi_{2}$};
\end{scope}

\draw [line width=0.5pt,dash pattern=on 1pt off 2pt,domain=-3:6.2] plot(\x,\x+1);
\draw [line width=0.5pt,dash pattern=on 1pt off 2pt,domain=-3:7.2] plot(\x,\x+3);

\draw [->,dash pattern=on 1pt off 4pt, line width=0.5pt] (2,3.65) -- (5.65,3.65);
\draw [anchor=north] (6.6,3.85) node {\tiny $Q_{1}\times Q_{2}$};

\end{tikzpicture}
\captionsetup{justification=centering}
\caption{\footnotesize The highlighted cubes $Q_{1} \times Q_{2}$ have sidelength $2^\ell$, with $\ell \in \mathbb{L}_{2}$, $\ell \sim 2^{k(b-1)}$}.  \label{figure:X2}
\end{figure}
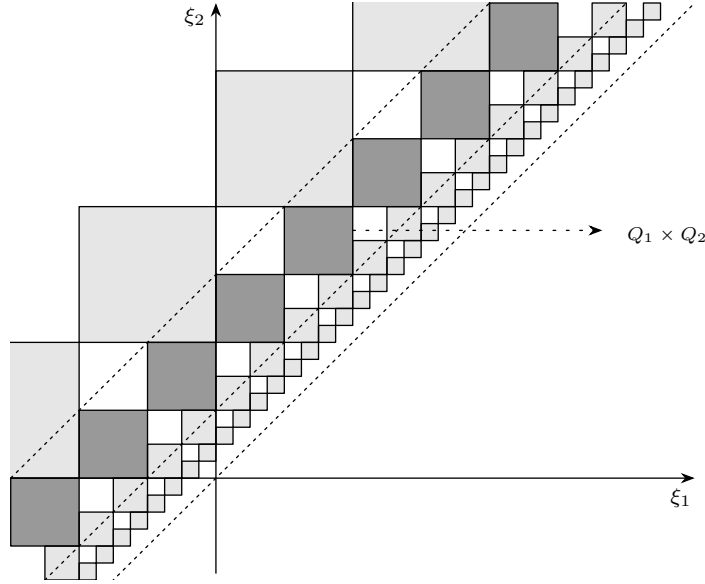

With this, our operator $B^{(2)}_{k, \ell}(f_1, f_2, f_3)$ becomes itself a superposition of terms of the form 
\begin{align*}
\tilde B^{(2)}_{k, \ell}(f_1, f_2, f_3)(x) := \sum_{Q_1 \times Q_2 \in \mathcal{W}_\ell}  \iiint_{\mathbb{R}^{3}} \widehat{f_{1}}(\xi_1)\widehat{f_{2}}(\xi_2)\widehat{f_{3}}(\xi_3)  \phi_{Q_1}(\xi_1) \, \phi_{Q_2}(\xi_2)\mathfrak{m}^{b}_{k}(\xi_3-\xi_2)  e^{2\pi ix(\xi_1+\xi_2+\xi_3)} \mathrm{d}\xi_1\mathrm{d}\xi_2\mathrm{d}\xi_3.
\end{align*}
Notice that $\mathcal{W}_\ell$ consists of square of side length comparable to $2^\ell$ that are at distance $\sim 2^\ell$ to the line $\Gamma_1$. This implies that the interval $Q_1$ is fully determined by the position of $Q_2$ and vice versa.

If $\ds \Pi_{Q_j}(f_j)$, for $j=1, 2$, denotes the smooth frequency projection of $f_j$ onto the interval $Q_j$, 
$$\ds \widehat{\Pi_{Q_j}(f_j)}(\eta)=\phi_{Q_j}(\eta)\cdot \widehat{f_{j}}(\eta),$$
then 
\begin{align*}
\tilde B^{(2)}_{k, \ell}(f_1, f_2, f_3)(x)= \sum_{Q_1 \times Q_2 \in \mathcal{W}_\ell} \Pi_{Q_1}(f_1)(x) \cdot \tilde T_k(\Pi_{Q_2}(f_2), f_3)(x),
\end{align*}
and thanks to Cauchy-Schwarz we can bound it from above by
\begin{align*}
\big| \tilde B^{(2)}_{k, \ell}(f_1, f_2, f_3)(x) \big| \leq \Big( \sum_{Q_1 \times Q_2 \in \mathcal{W}_\ell} \big| \Pi_{Q_1}(f_1)(x)\big|^2 \Big)^{\frac{1}{2}} \cdot  \Big( \sum_{Q_1 \times Q_2 \in \mathcal{W}_\ell} \big| \tilde T_k(\Pi_{Q_2}(f_2), f_3)(x)\big|^2 \Big)^\frac{1}{2}.
\end{align*}
For the first part, we can use the boundedness of the Rubio de Francia square function \cite{RF}, and for the second term we invoke $\ell^2$-valued extensions for the \emph{linear operator}
\[
f_2 \mapsto \tilde T_k(f_2, f_3),
\]
which are almost automatic thanks to Khintchine's inequality, once we know the boundedness of the $\tilde T_k$ operator (see for example \cite[Problem 8.2]{MS1}).

Thus we get
\begin{align*}
\big\| \tilde B^{(2)}_{k, \ell}(f_1, f_2, f_3)\big\| & \lesssim \big\|\big( \sum_{Q_1 \times Q_2 \in \mathcal{W}_\ell} \big| \Pi_{Q_1}(f_1)\big|^2 \big)^{\frac{1}{2}}\|_{p_1} \cdot \big\| \big( \sum_{Q_1 \times Q_2 \in \mathcal{W}_\ell} \big| \tilde T_k(\Pi_{Q_2}(f_2), f_3)\big|^2 \big)^\frac{1}{2} \|_{p_{2,3}} \\
&\lesssim 2^{-\delta b k} \|f_1\|_{p_1} \, \|f_2\|_{p_2} \, \| f_3\|_{p_3}, 
\end{align*}
provided $2 \leq p_1, p_2 < \infty$, $1<p_3 \leq \infty$ and $\ds \frac{1}{p_1}+\frac{1}{p_2}+\frac{1}{p_3}=\frac{1}{p}$.

\vspace{.1 cm }
\section{The oscillatory part $B^{(3)}_k$}
\label{sec:B_k^3}
In this section we study the operator
\begin{align*}
B^{(3)}_k(f_1,f_2,f_3)(x)=\sum_{\ell \in \mathbb{L}_3} \iiint_{\mathbb{R}^{3}} \widehat{f_{1}}(\xi_1)\widehat{f_{2}}(\xi_2)\widehat{f_{3}}(\xi_3) \psi\big(\frac{\xi_1-\xi_2}{2^\ell}\big)\mathfrak{m}^{b}_{k}(\xi_3-\xi_2)  e^{2\pi ix(\xi_1+\xi_2+\xi_3)} \mathrm{d}\xi_1\mathrm{d}\xi_2\mathrm{d}\xi_3,
\end{align*}
where $\ds \mathbb{L}_3:=\{\ell \in \Z : 2^{k(b-1)} 2^{N b k} \leq 2^\ell  \}$.

In this regime the scale of the $BHT$-type symbol dominates the scale of the $T_k$ operator, so we would like to replace (as before in Section \ref{sec:B_k^1}) the symbol $\ds \psi\big(\frac{\xi_1-\xi_2}{2^\ell}\big)$ by $\ds \psi\big(\frac{\xi_1-\frac{\xi_2+\xi_3}{2}}{2^\ell}\big)$. This allows to decouple the symbols and makes room for $BHT$ to be predominant.

\begin{figure}[ht]
\centering
\begin{tikzpicture}[line cap=round,line join=round,>=Stealth,x=1cm,y=1cm,decoration={brace,amplitude=5pt},scale=0.9]
\clip(-4.5,-1.5) rectangle (8.5,7);

%
%

\begin{scope}
\clip(-3,-1.5) rectangle (7,15);
\filldraw[line width=0.5pt,color=black,fill=black,fill opacity=0.4] (-4,0) -- (-4,2) -- (-2,2) -- (-2,0) -- cycle;
\filldraw[line width=0.5pt,color=black,fill=black,fill opacity=0.4] (-2,2) -- (-2,4) -- (0,4) -- (0,2) -- cycle;
\filldraw[line width=0.5pt,color=black,fill=black,fill opacity=0.4] (0,4) -- (0,6) -- (2,6) -- (2,4) -- cycle;
\filldraw[line width=0.5pt,color=black,fill=black,fill opacity=0.4] (2,6) -- (2,8) -- (4,8) -- (4,6) -- cycle;
\filldraw[line width=0.5pt,color=black,fill=black,fill opacity=0.4] (4,8) -- (4,10) -- (6,10) -- (6,8) -- cycle;

%
%

\filldraw[line width=0.5pt,color=black,fill=black,fill opacity=0.1] (-3,-1) -- (-3,0) -- (-2,0) -- (-2,-1) -- cycle;
\filldraw[line width=0.5pt,color=black,fill=black,fill opacity=0.1] (-2,0) -- (-2,1) -- (-1,1) -- (-1,0) -- cycle;
\filldraw[line width=0.5pt,color=black,fill=black,fill opacity=0.1] (-1,1) -- (-1,2) -- (0,2) -- (0,1) -- cycle;
\filldraw[line width=0.5pt,color=black,fill=black,fill opacity=0.1] (1,3) -- (1,4) -- (2,4) -- (2,3) -- cycle;
\filldraw[line width=0.5pt,color=black,fill=black,fill opacity=0.1] (0,2) -- (0,3) -- (1,3) -- (1,2) -- cycle;
\filldraw[line width=0.5pt,color=black,fill=black,fill opacity=0.1] (2,4) -- (2,5) -- (3,5) -- (3,4) -- cycle;
\filldraw[line width=0.5pt,color=black,fill=black,fill opacity=0.1] (3,5) -- (3,6) -- (4,6) -- (4,5) -- cycle;
\filldraw[line width=0.5pt,color=black,fill=black,fill opacity=0.1] (4,6) -- (4,7) -- (5,7) -- (5,6) -- cycle;

\filldraw[line width=0.5pt,color=black,fill=black,fill opacity=0.1] (1,2) -- (1,2.5) -- (1.5,2.5) -- (1.5,2) -- cycle;
\filldraw[line width=0.5pt,color=black,fill=black,fill opacity=0.1] (1.5,2.5) -- (1.5,3) -- (2,3) -- (2,2.5) -- cycle;
\filldraw[line width=0.5pt,color=black,fill=black,fill opacity=0.1] (2,3) -- (2,3.5) -- (2.5,3.5) -- (2.5,3) -- cycle;
\filldraw[line width=0.5pt,color=black,fill=black,fill opacity=0.1] (2.5,3.5) -- (2.5,4) -- (3,4) -- (3,3.5) -- cycle;
\filldraw[line width=0.5pt,color=black,fill=black,fill opacity=0.1] (1.5,2) -- (1.5,2.25) -- (1.75,2.25) -- (1.75,2) -- cycle;
\filldraw[line width=0.5pt,color=black,fill=black,fill opacity=0.1] (1.75,2.25) -- (1.75,2.5) -- (2,2.5) -- (2,2.25) -- cycle;
\filldraw[line width=0.5pt,color=black,fill=black,fill opacity=0.1] (2,2.5) -- (2,2.75) -- (2.25,2.75) -- (2.25,2.5) -- cycle;
\filldraw[line width=0.5pt,color=black,fill=black,fill opacity=0.1] (2.25,2.75) -- (2.25,3) -- (2.5,3) -- (2.5,2.75) -- cycle;
\filldraw[line width=0.5pt,color=black,fill=black,fill opacity=0.1] (2.5,3) -- (2.5,3.25) -- (2.75,3.25) -- (2.75,3) -- cycle;
\filldraw[line width=0.5pt,color=black,fill=black,fill opacity=0.1] (2.75,3.25) -- (2.75,3.5) -- (3,3.5) -- (3,3.25) -- cycle;
\filldraw[line width=0.5pt,color=black,fill=black,fill opacity=0.1] (3,3.5) -- (3,3.75) -- (3.25,3.75) -- (3.25,3.5) -- cycle;
\filldraw[line width=0.5pt,color=black,fill=black,fill opacity=0.1] (3.25,3.75) -- (3.25,4) -- (3.5,4) -- (3.5,3.75) -- cycle;
\filldraw[line width=0.5pt,color=black,fill=black,fill opacity=0.1] (-2.5,-1.5) -- (-2.5,-1) -- (-2,-1) -- (-2,-1.5) -- cycle;
\filldraw[line width=0.5pt,color=black,fill=black,fill opacity=0.1] (-2,-1) -- (-2,-0.5) -- (-1.5,-0.5) -- (-1.5,-1) -- cycle;
\filldraw[line width=0.5pt,color=black,fill=black,fill opacity=0.1] (-1.5,-0.5) -- (-1.5,0) -- (-1,0) -- (-1,-0.5) -- cycle;
\filldraw[line width=0.5pt,color=black,fill=black,fill opacity=0.1] (-1,0) -- (-1,0.5) -- (-0.5,0.5) -- (-0.5,0) -- cycle;
\filldraw[line width=0.5pt,color=black,fill=black,fill opacity=0.1] (-0.5,0.5) -- (-0.5,1) -- (0,1) -- (0,0.5) -- cycle;
\filldraw[line width=0.5pt,color=black,fill=black,fill opacity=0.1] (0,1) -- (0,1.5) -- (0.5,1.5) -- (0.5,1) -- cycle;
\filldraw[line width=0.5pt,color=black,fill=black,fill opacity=0.1] (0.5,1.5) -- (0.5,2) -- (1,2) -- (1,1.5) -- cycle;
\filldraw[line width=0.5pt,color=black,fill=black,fill opacity=0.1] (3,4) -- (3,4.5) -- (3.5,4.5) -- (3.5,4) -- cycle;
\filldraw[line width=0.5pt,color=black,fill=black,fill opacity=0.1] (3.5,4.5) -- (3.5,5) -- (4,5) -- (4,4.5) -- cycle;
\filldraw[line width=0.5pt,color=black,fill=black,fill opacity=0.1] (4,5) -- (4,5.5) -- (4.5,5.5) -- (4.5,5) -- cycle;
\filldraw[line width=0.5pt,color=black,fill=black,fill opacity=0.1] (4.5,5.5) -- (4.5,6) -- (5,6) -- (5,5.5) -- cycle;
\filldraw[line width=0.5pt,color=black,fill=black,fill opacity=0.1] (5,6) -- (5,6.5) -- (5.5,6.5) -- (5.5,6) -- cycle;
\filldraw[line width=0.5pt,color=black,fill=black,fill opacity=0.1] (5.5,6.5) -- (5.5,7) -- (6,7) -- (6,6.5) -- cycle;

\filldraw[line width=0.5pt,color=black,fill=black,fill opacity=0.1] (-2,-1.5) -- (-2,-1.25) -- (-1.75,-1.25) -- (-1.75,-1.5) -- cycle;
\filldraw[line width=0.5pt,color=black,fill=black,fill opacity=0.1] (-1.75,-1.25) -- (-1.75,-1) -- (-1.5,-1) -- (-1.5,-1.25) -- cycle;
\filldraw[line width=0.5pt,color=black,fill=black,fill opacity=0.1] (-1.5,-1) -- (-1.5,-0.75) -- (-1.25,-0.75) -- (-1.25,-1) -- cycle;
\filldraw[line width=0.5pt,color=black,fill=black,fill opacity=0.1] (-1.25,-0.75) -- (-1.25,-0.5) -- (-1,-0.5) -- (-1,-0.75) -- cycle;
\filldraw[line width=0.5pt,color=black,fill=black,fill opacity=0.1] (-1,-0.5) -- (-1,-0.25) -- (-0.75,-0.25) -- (-0.75,-0.5) -- cycle;
\filldraw[line width=0.5pt,color=black,fill=black,fill opacity=0.1] (-0.75,-0.25) -- (-0.75,0) -- (-0.5,0) -- (-0.5,-0.25) -- cycle;
\filldraw[line width=0.5pt,color=black,fill=black,fill opacity=0.1] (-0.5,0) -- (-0.5,0.25) -- (-0.25,0.25) -- (-0.25,0) -- cycle;
\filldraw[line width=0.5pt,color=black,fill=black,fill opacity=0.1] (-0.25,0.25) -- (-0.25,0.5) -- (0,0.5) -- (0,0.25) -- cycle;
\filldraw[line width=0.5pt,color=black,fill=black,fill opacity=0.1] (0,0.5) -- (0,0.75) -- (0.25,0.75) -- (0.25,0.5) -- cycle;
\filldraw[line width=0.5pt,color=black,fill=black,fill opacity=0.1] (0.25,0.75) -- (0.25,1) -- (0.5,1) -- (0.5,0.75) -- cycle;
\filldraw[line width=0.5pt,color=black,fill=black,fill opacity=0.1] (0.5,1) -- (0.5,1.25) -- (0.75,1.25) -- (0.75,1) -- cycle;
\filldraw[line width=0.5pt,color=black,fill=black,fill opacity=0.1] (0.75,1.25) -- (0.75,1.5) -- (1,1.5) -- (1,1.25) -- cycle;
\filldraw[line width=0.5pt,color=black,fill=black,fill opacity=0.1] (1,1.5) -- (1,1.75) -- (1.25,1.75) -- (1.25,1.5) -- cycle;
\filldraw[line width=0.5pt,color=black,fill=black,fill opacity=0.1] (1.25,1.75) -- (1.25,2) -- (1.5,2) -- (1.5,1.75) -- cycle;
\filldraw[line width=0.5pt,color=black,fill=black,fill opacity=0.1] (3.5,4) -- (3.5,4.25) -- (3.75,4.25) -- (3.75,4) -- cycle;
\filldraw[line width=0.5pt,color=black,fill=black,fill opacity=0.1] (3.75,4.25) -- (3.75,4.5) -- (4,4.5) -- (4,4.25) -- cycle;
\filldraw[line width=0.5pt,color=black,fill=black,fill opacity=0.1] (4,4.5) -- (4,4.75) -- (4.25,4.75) -- (4.25,4.5) -- cycle;
\filldraw[line width=0.5pt,color=black,fill=black,fill opacity=0.1] (4.25,4.75) -- (4.25,5) -- (4.5,5) -- (4.5,4.75) -- cycle;
\filldraw[line width=0.5pt,color=black,fill=black,fill opacity=0.1] (4.5,5) -- (4.5,5.25) -- (4.75,5.25) -- (4.75,5) -- cycle;
\filldraw[line width=0.5pt,color=black,fill=black,fill opacity=0.1] (4.75,5.25) -- (4.75,5.5) -- (5,5.5) -- (5,5.25) -- cycle;
\filldraw[line width=0.5pt,color=black,fill=black,fill opacity=0.1] (5,5.5) -- (5,5.75) -- (5.25,5.75) -- (5.25,5.5) -- cycle;
\filldraw[line width=0.5pt,color=black,fill=black,fill opacity=0.1] (5.25,5.75) -- (5.25,6) -- (5.5,6) -- (5.5,5.75) -- cycle;
\filldraw[line width=0.5pt,color=black,fill=black,fill opacity=0.1] (5.5,6) -- (5.5,6.25) -- (5.75,6.25) -- (5.75,6) -- cycle;
\filldraw[line width=0.5pt,color=black,fill=black,fill opacity=0.1] (5.75,6.25) -- (5.75,6.5) -- (6,6.5) -- (6,6.25) -- cycle;
\filldraw[line width=0.5pt,color=black,fill=black,fill opacity=0.1] (6,6.5) -- (6,6.75) -- (6.25,6.75) -- (6.25,6.5) -- cycle;
\filldraw[line width=0.5pt,color=black,fill=black,fill opacity=0.1] (6.25,6.75) -- (6.25,7) -- (6.5,7) -- (6.5,6.75) -- cycle;

\draw [line width=0.5pt,dash pattern=on 1pt off 2pt,domain=-5:12] plot(\x,\x);
\draw [->,line width=0.5pt] (-3,0) -- (7,0);
\draw [->,line width=0.5pt] (0,-1.4) -- (0,7);
\draw [anchor=north] (6.8,0) node {\scriptsize $\xi_{1}$};
\draw [anchor= east] (0,6.8) node {\scriptsize $\xi_{2}$};
\end{scope}

\draw [line width=0.5pt,dash pattern=on 1pt off 2pt,domain=-3:6.2] plot(\x,\x+1);
\draw [line width=0.5pt,dash pattern=on 1pt off 2pt,domain=-3:7.2] plot(\x,\x+3);

\draw [->,dash pattern=on 1pt off 4pt, line width=0.5pt] (0,3.65) -- (5.65,3.65);
\draw [anchor=north] (6.6,3.85) node {\tiny $Q_{1}\times Q_{2}$};

\end{tikzpicture}
\captionsetup{justification=centering}
\caption{\footnotesize The highlighted cubes $Q_{1}\times Q_{2}$ have sidelength $2^\ell \geq 2^{-k} 2^{(N+1)bk}$, with $ \ell \in \mathbb{L}_{3}$.} \label{figure:X3}
\end{figure}
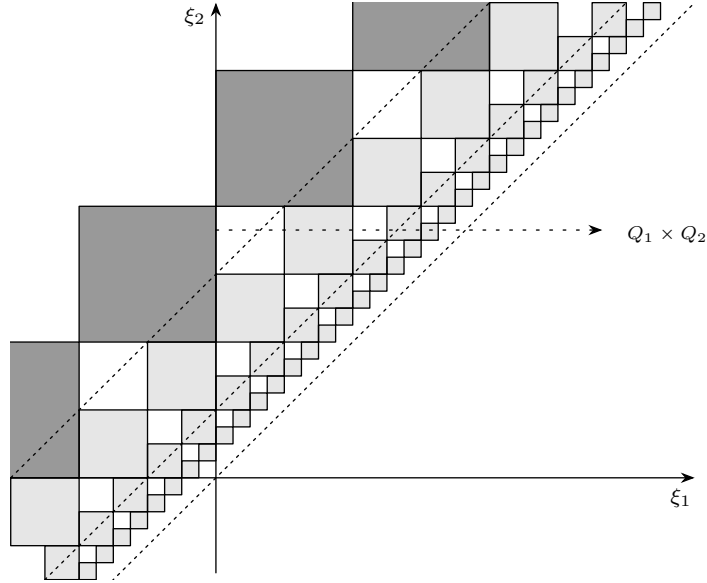

By the Fundamental Theorem of Calculus, 
\begin{align*}
\psi_\ell(\xi_1-\xi_2)= \psi_\ell\big(\xi_1-\frac{\xi_2+\xi_3}{2}\big) +\int_0^1 \tilde \psi_{\ell} \big(  \xi_1  -\frac{\xi_2+\xi_3}{2} +\tau \cdot \frac{\xi_3-\xi_2}{2} \big) \, 2^{-\ell} \cdot \frac{\xi_3-\xi_2}{2} \mathrm{d}\tau, 
\end{align*}
where 
\begin{align*}
\tilde \psi_{\ell}(\zeta):= \psi' \big( \frac{\zeta}{2^\ell}\big)
\end{align*}
has the same properties as $\tilde \psi_{\ell}$ functions.

With this, $B^{(3)}_k$ splits into two parts called $\mathfrak{M}^{(3)}_k$ and $\mathfrak{R}^{(3)}_k$, respectively.

\subsection{The operator $\mathfrak{M}^{(3)}_k$} The main part is given by
\begin{align*}
\mathfrak{M}^{(3)}_k(f_1,f_2,f_3)(x)=\sum_{\ell \in \mathbb{L}_3} \iiint_{\mathbb{R}^{3}} \widehat{f_{1}}(\xi_1)\widehat{f_{2}}(\xi_2)\widehat{f_{3}}(\xi_3) \psi_\ell\big(\xi_1-\frac{\xi_2+\xi_3}{2}\big) \mathfrak{m}^{b}_{k}(\xi_3-\xi_2)  e^{2\pi ix(\xi_1+\xi_2+\xi_3)} \mathrm{d}\xi_1\mathrm{d}\xi_2\mathrm{d}\xi_3.
\end{align*}
We recognize this as being
\begin{align*}
BHT_{\mathbb{L}_3}(f_1, T_k (f_2, f_3))(x),
\end{align*}
where $T_k$ was defined in \eqref{def:T_k:decay} and
\begin{align*}
BHT_{\mathbb{L}_3}(F_1, F_2)(x):= \sum_{\ell \in \mathbb{L}_3} \iint_{\mathbb{R}^{2}} \widehat{F_{1}}(\eta_1)\widehat{F_{2}}(\eta_2) \psi_\ell\big(\eta_1-\frac{\eta_2}{2}\big) e^{2\pi ix(\eta_1+\eta_2)} \mathrm{d}\eta_1\mathrm{d}\eta_2.
\end{align*}
Again due to its dilation invariance, $BHT_{\mathbb{L}_3}$ satisfies the same type of $L^p$ estimates as
\begin{align*}
BHT_{\geq 0}(F_1, F_2)(x):= \sum_{\ell \geq 0} \iint_{\mathbb{R}^{2}} \widehat{F_{1}}(\eta_1)\widehat{F_{2}}(\eta_2) \psi_\ell\big(\eta_1-\frac{\eta_2}{2}\big) e^{2\pi ix(\eta_1+\eta_2)} \mathrm{d}\eta_1\mathrm{d}\eta_2,
\end{align*}
uniformly\footnote{The $k \in \N^*$ parameter appears implicitly in the definition of $\mathbb L_3$, but the bound for $BHT_{\mathbb{L}_3}$ are independent of it.} in $k \in \N^{*}$.

Thanks to the boundedness of $BHT_{\mathbb{L}_3}$ and $T_k$ from Theorem \ref{thm:conseq:BCa},  we have that 
\begin{align*}
\big\|\mathfrak{M}^{(3)}_k (f_1, f_2, f_3)  \big\|_p \lesssim 2^{-\delta b k}  \|f_1\|_{p_1} \, \|f_2\|_{p_2} \, \|f_3\|_{p_3},
\end{align*}
as long as $\ds 1<p_1, p_2, p_3 \leq \infty$, $1 \leq p < \infty$, $\frac{1}{p_1}+\frac{1}{p_2}+\frac{1}{p_3}=\frac{1}{p}$ and $\frac{1}{p_2}+\frac{1}{p_3}<1$. The last constraint comes from the fact that $T_k (f_2, f_3)$ is an input for the bilinear operator $BHT_{\mathbb{L}_3}$, but given the other constraints this means simply that $(p_1, p_2, p_3) \neq (\infty, p_2, p_2')$.

\subsection{The remainder $\ds \mathfrak{R}^{(3)}_k$} The residual term can be written as
\begin{equation*}
    \eqalign{
    \displaystyle\mathfrak{R}^{(3)}_k(f_1,f_2,f_3)(x) &=\displaystyle \int_0^1 \sum_{\ell \in \mathbb{L}_3} 2^{-\ell} \iiint_{\mathbb{R}^{3}} \widehat{f_{1}}(\xi_1)\widehat{f_{2}}(\xi_2)\widehat{f_{3}}(\xi_3) \tilde \psi_{\ell} \big(  \xi_1  -\frac{\xi_2+\xi_3}{2} +\tau \cdot \frac{\xi_3-\xi_2}{2} \big)  \cdot (\xi_3-\xi_2) \cr
    &\qquad\qquad\qquad\qquad\qquad\qquad\qquad\times \,\mathfrak{m}^{b}_{k}(\xi_3-\xi_2)  e^{2\pi ix(\xi_1+\xi_2+\xi_3)} \mathrm{d}\xi_1\mathrm{d}\xi_2\mathrm{d}\xi_3 \mathrm{d} \tau \cr
    }
\end{equation*}

We don't have enough frequency information on $\xi_3-\xi_2$. At this point, we could use the asymptotic expansion for the oscillatory integral in order to obtain a more precise localization of the functions - in particular, up to error terms, we can assume that $(\xi_2, \xi_3)$ belongs to the strip $\ds \{ (\xi_2, \xi_3) : |\xi_3 - \xi_2| \sim 2^{k(b-1)}\}$.

We can avoid this extra decomposition by reworking the symbol $\ds  (\xi_3-\xi_2) \,\mathfrak{m}^{b}_{k}(\xi_3-\xi_2)$, which is given by
\begin{align*}
\int_{\R} (\xi_3-\xi_2) \, e^{-2 \pi i 2^k s (\xi_3-\xi_2)}  e^{2\pi i 2^{kb}s^{b}}\rho(s) \mathrm{d}s.
\end{align*}
After integrating by parts, this becomes
\begin{align*}
\int_{\R}  e^{-2 \pi i 2^k s (\xi_3-\xi_2)}  e^{2\pi i 2^{kb}s^{b}} \big( b \, 2^{k(b-1)} s^{b-1} \rho(s) +\frac{1}{2\pi i} 2^{-k} \rho'(s)  \big) \mathrm{d}s:= \mathfrak{r}^{b}_{1,k}(\xi_3-\xi_2)+\mathfrak{r}^{b}_{2,k}(\xi_3-\xi_2),
\end{align*}
where 
\begin{align*}
\mathfrak{r}^{b}_{1,k}(\xi_3-\xi_2):=\int_{\R}  e^{-2 \pi i 2^k s (\xi_3-\xi_2)}  e^{2\pi i 2^{kb}s^{b}}  b \, 2^{k(b-1)} s^{b-1} \rho(s) \mathrm{d}s
\end{align*}
and
\begin{align*}
\mathfrak{r}^{b}_{2,k}(\xi_3-\xi_2):=\int_{\R}  e^{-2 \pi i 2^k s (\xi_3-\xi_2)}  e^{2\pi i 2^{kb}s^{b}} \frac{1}{2\pi i} 2^{-k} \rho'(s) \mathrm{d}s.
\end{align*}
We can see that both symbols resemble $\mathfrak{m}^{b}_{k}(\xi_3-\xi_2)$, although the term $\mathfrak{r}^{b}_{1,k}(\xi_3-\xi_2)$ behaves worse due to the derivative that hit the highly oscillating complex exponential.

We then further split $\mathfrak{R}^{(3)}_k$ as $\ds \mathfrak{R}^{(3)}_k = \mathfrak{R}^{(3)}_{1,k}+\mathfrak{R}^{(3)}_{2, k}$, where $\mathfrak{R}^{(3)}_{j,k}$ corresponds to the symbol $\mathfrak{r}^{b}_{j,k}$ for $\ds j \in \{1,2\}$. As before, we pull out the integrals in $s$ and $\tau$ to obtain 
\begin{align*}
\mathfrak{R}^{(3)}_{1,k}(f_1, f_2, f_3)(x)=\int_0^1 \int_{[{1 \over 2}, 2]} e^{2\pi i 2^{kb}s^{b}} s^{b-1} \rho(s) \sum_{\ell \in \mathbb{L}_3} 2^{-\ell} 2^{k(b-1)} \iiint_{\mathbb{R}^{3}} \widehat{f_{1}}(\xi_1) \widehat{f_{2}}(\xi_2)  e^{2 \pi i 2^k s \xi_2} \widehat{f_{3}}(\xi_3) e^{-2 \pi i 2^k s \xi_3} & \\
 \tilde \psi_{\ell} \big(  \xi_1  -\frac{1+\tau}{2} \xi_2 - \frac{1-\tau}{2} \xi_3 \big) e^{2\pi ix(\xi_1+\xi_2+\xi_3)}  \mathrm{d}\xi_1\mathrm{d}\xi_2\mathrm{d}\xi_3 \mathrm{d}s \mathrm{d} \tau.\quad &
\end{align*}
Already we see that we have decay with respect to $\ell \in  \mathbb{L}_3$, and since $2^\ell \geq 2^{k(b-1)} 2^{N b k}$, we have overall decay in $2^{k}$. So it is enough to bound without decay, uniformly in $\tau, s, \ell$ the trilinear operator
\begin{align}
\label{eq:tril:H}
(f_1, f_2, f_3) \mapsto \iiint_{\mathbb{R}^{3}} \widehat{f_{1}}(\xi_1) \widehat{f_{2}}(\xi_2)  e^{2 \pi i 2^k s \xi_2} \widehat{f_{3}}(\xi_3) e^{-2 \pi i 2^k s \xi_3} \tilde \psi_{\ell} \big(  \xi_1  -\frac{1+\tau}{2} \xi_2 - \frac{1-\tau}{2} \xi_3 \big) e^{2\pi ix(\xi_1+\xi_2+\xi_3)}  \mathrm{d}\xi_1\mathrm{d}\xi_2\mathrm{d}\xi_3.
\end{align}
This corresponds to a one-scale version of the trilinear Hilbert transform. Although the boundedness of the trilinear Hilbert transform remains a widely open question, and although it is known that the associated one-scale operator cannot be bounded when all the input functions are close to $L^1$ (see \cite{d}, \cite{MS2}), the above operator can be bounded.

Indeed, let $\alpha_1, \alpha_2, \alpha_3 \in \R$ and let $\Phi$ be an $L^1$-normalized function. Define  
\begin{align*}
H_{\alpha_1, \alpha_2, \alpha_3}(F_1, F_2, F_3)(x)= \int_{\R} F_1(x+\alpha_1 \, t) \, F_2(x+\alpha_2 \, t) \, F_3(x+\alpha_3 \, t) \, \Phi(t) \mathrm{d} t.
\end{align*}
For $1<q_1, q_2, q_3 \leq \infty $ such that $\ds \frac{1}{q_1}+\frac{1}{q_2}+\frac{1}{q_3}=1$, we have the pointwise inequality
\begin{align*}
\big| H_{\alpha_1, \alpha_2, \alpha_3}(F_1, F_2, F_3)(x)\big| \lesssim \prod_{j=1}^3 \Big(  \int_{\R} \big(|F_j(x+\alpha_j \, t)| \, | \Phi(t)|^\frac{1}{q_j}\big)^{q_j}  \mathrm{d} t \Big)^{\frac{1}{q_j}}.
\end{align*}
This implies for any $1<p_1, p_2, p_3 \leq \infty$, $1 \leq p <\infty$ such that $\ds \frac{1}{p_1}+\frac{1}{p_2}+\frac{1}{p_3}=\frac{1}{p}$, and for any $q_1, q_2, q_3$ as above, that
\begin{align*}
\|H_{\alpha_1, \alpha_2, \alpha_3}(F_1, F_2, F_3)\|_p \lesssim \prod_{j=1}^3 \Big( \Big(  \int_{\R} \big(|F_j(x+\alpha_j \, t)| \, | \Phi(t)|^\frac{1}{q_j}\big)^{q_j}  \mathrm{d} t \Big)^{\frac{p_j}{q_j}}      \mathrm{d} x \Big)^{\frac{1}{p_j}}.
\end{align*}

Now we choose $q_1, q_2, q_3$ so that $1 \leq q_j \leq p_j$ for all $1 \leq j \leq 3$, which is possible since
\[
 \frac{1}{p_1}+\frac{1}{p_2}+\frac{1}{p_3}=\frac{1}{p} \leq 1=\frac{1}{q_1}+\frac{1}{q_2}+\frac{1}{q_3}.
\]
This allows us to apply Minkowski's integral inequality and switch the order of integration in $x$ and $t$:
\begin{align*}
\|H_{\alpha_1, \alpha_2, \alpha_3}(F_1, F_2, F_3)\|_p & \lesssim \prod_{j=1}^3 \Big( \Big(  \int_{\R} \big(|F_j(x+\alpha_j \, t)| \, | \Phi(t)|^\frac{1}{q_j}\big)^{p_j}  \mathrm{d}x \Big)^{\frac{q_j}{p_j}}      \mathrm{d}t \Big)^{\frac{1}{q_j}} \\
&\lesssim \big( \prod_{j=1}^3 \|F_j\|_{p_j} \big) \|\Phi\|_1.
\end{align*}
Notice that these bounds are \emph{uniform} in $\alpha_1, \alpha_2, \alpha_3 \in \R$.

Now we apply these estimates for $H_{\alpha_1, \alpha_2, \alpha_3}$ to the functions
\[
F_1(x):=f_1(x), \qquad F_2(x):= f_2(x+ 2^ks), \qquad F_3(x):= f_3(x- 2^ks),
\]
for $\ds (\alpha_1, \alpha_2, \alpha_3)= (-1, \frac{1+\tau}{2},  \frac{1-\tau}{2})$, and with the $L^1$ normalized function $\ds \check{\tilde \psi}_{\ell}$.
This yields that the operator in \eqref{eq:tril:H} is bounded, uniformly in $\ell, k, s$ and $\tau$, for any $1<p_1, p_2, p_3 \leq \infty$, $1 \leq p <\infty$ such that $\ds \frac{1}{p_1}+\frac{1}{p_2}+\frac{1}{p_3}=\frac{1}{p}$. Thus we get, within the same range, that
\begin{align*}
\|\mathfrak{R}^{(3)}_{1,k}(f_1, f_2, f_3)\|_p & \lesssim  \sum_{\ell \in \mathbb{L}_3} 2^{-\ell} 2^{k(b-1)} \|f_1\|_{p_1} \, \|f_2\|_{p_2} \, \|f_3\|_{p_3} \\
&\lesssim 2^{-N bk}  \|f_1\|_{p_1} \, \|f_2\|_{p_2} \, \|f_3\|_{p_3}.
\end{align*}

In a similar way, one can prove the same type of $L^p$ estimates for $\mathfrak{R}^{(3)}_{2,k}(f_1, f_2, f_3)$. Because the symbol $\mathfrak{r}^{b}_{2,k}$ behaves better that $\mathfrak{r}^{b}_{1,k}$, we have even more decay in $2^k$.

\section{Final interpolation}\label{sec:finalint}

Now we recall all the estimates we have proved for each piece of our operator $B$ obtained after decomposing it in Section \ref{sec:initial:red:space}. Recall that
\begin{align*}
B(f_1, f_2, f_3)(x) &=B_0(f_1, f_2, f_3)+\sum_{k \geq 1} B_k^{(1)}(f_1, f_2, f_3)(x)+B_k^{(2)}(f_1, f_2, f_3)(x) +B_k^{(3)}(f_1, f_2, f_3)(x) \\
&=B_0(f_1, f_2, f_3)+\sum_{k \geq 1} \mathfrak{M}_k^{(1)}(f_1, f_2, f_3)(x)+\mathfrak{R}_k^{(1)}(f_1, f_2, f_3)(x) \\
+&\sum_{k \geq 1} \sum_{\ell \in \mathbb{L}_2} B_{k, \ell}^{(2)}(f_1, f_2, f_3)(x) + \sum_{k \geq 1} \mathfrak{M}_k^{(3)}(f_1, f_2, f_3)(x)+\mathfrak{R}_k^{(3)}(f_1, f_2, f_3)(x).
\end{align*}

The chart below collects all conditions and constraints that we currently have for the $L^p$ boundedness of each of the components of $B$. We study their $\ds L^{p_1} \times L^{p_2} \times L^{p_3} \to L^{p}$ boundedness under the H\"older condition
\[
\frac{1}{p_1}+\frac{1}{p_2}+\frac{1}{p_3}=\frac{1}{p},
\]
with $1<p_1, p_2, p_3 \leq \infty$, $1/3 < p < \infty$. As previously mentioned, we are restricted to the Banach range, hence $1\leq p \leq \infty$.

\begin{center}
\begin{tabularx}{0.99\textwidth} { 
  | >{\raggedright\arraybackslash}X 
  | >{\centering\arraybackslash}X 
  | >{\raggedleft\arraybackslash}X | }
 \hline
 Boundedness of $B_0$ & $\ds \|B_0(f_1, f_2, f_3)\|_p \lesssim \prod_{j=1}^3 \|f_j\|_{p_j}$ for all $1 < p_1, p_2, p_3 \leq \infty$, $1 \leq p <\infty$.  \\
 \hline
  Boundedness of $B_k$ without decay  & $\ds \|B_k(f_1, f_2, f_3)\|_p \lesssim \prod_{j=1}^3 \|f_j\|_{p_j}$ for all $1 < p_1, p_2, p_3 \leq \infty$, $1 \leq p <\infty$.   \\
 \hline
Boundedness of $\mathfrak{M}_k^{(1)}$ with decay  & $\ds \|\mathfrak{M}_k^{(1)}(f_1, f_2, f_3)\|_p \lesssim 2^{-\delta b k} \prod_{j=1}^3 \|f_j\|_{p_j}$ for all $1 < p_1, p_2, p_3 \leq \infty$, $1 \leq p <\infty$ but we cannot have $(p_1, p_2, p_3)=(p_1, p_1', \infty)$.  \\
  \hline
Boundedness of $\mathfrak{R}_k^{(1)}$ with decay  & $\ds \big\|\mathfrak{R}_k^{(1)}(f_1, f_2, f_3)\big\|_p \lesssim 2^{-kb (N-1)} \prod_{j=1}^3 \|f_j\|_{p_j}$ for all  $1 < p_1, p_2, p_3 \leq \infty$, $1 \leq p <\infty$. \\
  \hline
For every $\ell \in \mathbb{L}_2=\{ \ell \in \Z : k(b-1)- N b k \leq \ell \leq k(b-1)  N b k \}$, boundedness of $B_{k, \ell}^{(2)}$ with decay  & $\ds \|B_{k, \ell}^{(2)}(f_1, f_2, f_3)\|_p \lesssim 2^{-\delta b k} \prod_{j=1}^3 \|f_j\|_{p_j}$ for all  $ 2 \leq p_1, p_2<\infty, 1< p_3 \leq \infty$, $1 \leq p <\infty$.  \\
  \hline
The above implies the boundedness of $B_k^{(2)}$ with decay  & $\ds \|B_k^{(2)}(f_1, f_2, f_3)\|_p \lesssim 2Nbk 2^{-\delta b k} \prod_{j=1}^3 \|f_j\|_{p_j}$ for all $ 2 \leq p_1, p_2<\infty, 1< p_3 \leq \infty$, $1 \leq p <\infty$.  \\  
\hline
Boundedness of $\mathfrak{M}_k^{(3)}$ with decay  & $\ds \|\mathfrak{M}_k^{(3)}(f_1, f_2, f_3)\|_p \lesssim 2^{-\delta b k} \prod_{j=1}^3 \|f_j\|_{p_j}$ for all $1 < p_1, p_2, p_3 \leq \infty$, $1 \leq p <\infty$ but we cannot have $(p_1, p_2, p_3)=(\infty, p_2, p_2')$.  \\
  \hline 
Boundedness of $\mathfrak{R}_k^{(3)}$ with decay  & $\ds \big\|\mathfrak{R}_k^{(3)}(f_1, f_2, f_3)\big\|_p \lesssim 2^{-kb N} \prod_{j=1}^3 \|f_j\|_{p_j}$ for all $ 1 < p_1, p_2, p_3 \leq \infty$, $1 \leq p <\infty$.  \\
  \hline 
\end{tabularx}
\end{center}
We recall that the parameter $\delta>0$ appearing in the boundedness of $\ds \mathfrak{M}_k^{(1)}, B_{k, \ell}^{(2)}$ (for $\ell \in \mathbb{L}_2$) and $\mathfrak{M}_k^{(3)}$ comes from Theorem \ref{thm:conseq:BCa}. We lose a bit of this decay (it is a logarithmic loss) since we need to take into account the cardinality of the set $\mathbb{L}_2$. For the terms $\mathfrak{R}_k^{(1)}$ and $\mathfrak{R}_k^{(3)}$, the decay is given by $\ds 2^{-kb(N-1)}$, so if we choose $N=2$, we get 
\begin{align}
\label{eq:decay:B_k:v:local}
\|B_k(f_1, f_2, f_3)\|_p \lesssim k 2^{-\delta b k} \prod_{j=1}^3 \|f_j\|_{p_j}
\end{align}
for all $ 2 \leq p_1, p_2<\infty, 1< p_3 \leq \infty$, $1 \leq p <\infty$.

We interpolate this with the estimates for $B_k$ without decay that were shown in Section \ref{sec:B_k:no:decay}, to deduce the result in Theorem \ref{thm:intermediary:osc:k:fixed}. In other words, we propagate via interpolation the decay from \eqref{eq:decay:B_k:v:local} to the whole Banach range (as described in Theorem \ref{thm:intermediary:osc:k:fixed}).

Now we can finally sum up the $B_k$ terms, and invoke the boundedness of $B_0$, to obtain the result in Theorem \ref{mainthm}.




\Addresses

\end{document}